\documentclass[twoside,11pt]{article}
\usepackage[titletoc]{appendix}

\usepackage{jmlr2e}
\usepackage{amsmath}
\usepackage{color}
\usepackage{amsfonts}
\usepackage{hyperref}
\usepackage{bbm}
\usepackage{booktabs} 
\usepackage{algorithm}
\usepackage{algorithmic}

\makeatletter
\newcommand\backmatter{
    \def\chaptermark##1{\markboth{%
        \ifnum  \c@secnumdepth > \m@ne  \@chapapp\ \thechapter:  \fi  ##1}{%
        \ifnum  \c@secnumdepth > \m@ne  \@chapapp\ \thechapter:  \fi  ##1}}%
    \def\sectionmark##1{\relax}}
\makeatother

\newcommand{\bx}{{\mathbf x}}
\newcommand{\by}{{\mathbf y}}
\newcommand{\bz}{{\mathbf z}}
\newcommand{\bw}{{\mathbf w}}
\newcommand{\bb}{{\mathbf  b}}

\newcommand{\bu}{{\mathbf u}}

\newcommand{\be}{{\mathbf e}}
\newcommand{\br}{{\mathbf{r}}}

\newcommand{\balpha}{\boldsymbol{\alpha}}

\newcommand{\bE}{\mathbb{E}}
\newcommand{\reals}{\mathbb{R}}

\newcommand{\bN}{\mathbb{N}}

\newcommand{\bC}{\mathbb{C}}

\newcommand{\spec}[1]{\sigma\!\circpar{#1}}

\newcommand{\eqdef}{\stackrel{\vartriangle}{=}}

\newcommand{\cA}{\mathcal{A}}

\newcommand{\cC}{\mathcal{C}}

\newcommand{\cF}{\mathcal{F}}

\newcommand{\cI}{\mathcal{I}}

\newcommand{\cL}{\mathcal{L}}
\newcommand{\cM}{\mathcal{M}}

\newcommand{\cO}{\mathcal{O}}
\newcommand{\cP}{\mathcal{P}}
\newcommand{\cQ}{\mathcal{Q}}

\newcommand{\cS}{\mathcal{S}}

\newcommand{\cU}{\mathcal{U}}

\newcommand{\cX}{\mathcal{X}}

\DeclareMathOperator*{\argmin}{argmin} 

\renewcommand{\eqref}[1]{Equation~(\ref{#1})}
\newcommand{\ineqref}[1]{Inequality~(\ref{#1})}

\newcommand{\figref}[1]{Figure~\ref{#1}}

\newcommand{\thmref}[1]{Theorem~\ref{#1}}
\newcommand{\lemref}[1]{Lemma~\ref{#1}}
\newcommand{\defref}[1]{Definition~\ref{#1}}
\newcommand{\corref}[1]{Corollary~\ref{#1}}

\newcommand{\norm}[1]{\left\Vert#1\right\Vert}
\newcommand{\normsq}[1]{\left\Vert#1\right\Vert^2}
\newcommand{\inprod}[2]{ \left< #1 , #2 \right>}
\newcommand{\circpar}[1]{\left( #1 \right)}
\newcommand{\rectpar}[1]{\left[ #1 \right]}
\newcommand{\absval}[1]{\left|\,#1\,\right|}

\newcommand{\set}[1]{\left\{ #1 \right\} }
\newcommand{\myset}[2]{\left\{ #1 \right.\left|~ #2 \right\} }

\newcommand{\Diag}[1]{\text{Diag}\circpar{#1}}

\newcommand{\posdef}[2]{\cS^{#1}\!\left({#2}\right)}
\newcommand{\posdefun}[2]{\cQ^{#1}\!\left( #2 \right)}
\newcommand{\quadab}[1]{f_{#1}(\bx)}
\newcommand{\syspol}[2]{\cL_{#1}(#2)}
\newcommand{\syspola}{\syspol{}{\lambda,X}}

\newcommand{\mymat}[1]{\circpar{\begin{array}{cccccccc} #1 \end{array}}}

\newcommand{\bigO}[1]{\mathcal{O}{\left(#1\right)}}
\newcommand{\bigtO}[1]{\tilde{\mathcal{O}}{\left(#1\right)}}

\newcommand{\chiM}{\chi_{\cM}}

\newcommand{\IC}{\cI\cC}

\begin{document}
\title{On Lower and Upper Bounds for Smooth and Strongly Convex Optimization Problems}

\author{\name Yossi Arjevani \email yossi.arjevani@weizmann.ac.il \\
       \addr Department of Computer Science and Applied Mathematics\\
       Weizmann Institute of Science\\
       Rehovot 7610001, Israel
       \AND
       \name Shai Shalev-Shwartz  \email shais@cs.huji.ac.il \\
       \addr School of Computer Science and Engineering\\
       The Hebrew University\\
       Givat Ram, Jerusalem 9190401, Israel
			 \AND
       \name Ohad Shamir  \email ohad.shamir@weizmann.ac.il  \\
      \addr Department of Computer Science and Applied Mathematics\\
       Weizmann Institute of Science\\
       Rehovot 7610001, Israel}
\editor{}
\maketitle


\begin{abstract}
We develop a novel framework to study smooth and strongly convex optimization algorithms, both deterministic and stochastic. Focusing on quadratic functions we are able to examine optimization algorithms as a recursive application of linear operators. This, in turn, reveals a powerful connection between a class of optimization algorithms and the analytic theory of polynomials whereby new lower and upper bounds are derived. Whereas existing lower bounds for this setting are only valid when the dimensionality scales with the number of iterations, our lower bound holds in the natural regime where the dimensionality is fixed. Lastly, expressing it as an optimal solution for the corresponding optimization problem over polynomials, as formulated by our framework, we present a novel systematic derivation of Nesterov's well-known Accelerated Gradient Descent method. This rather natural interpretation of AGD contrasts with earlier ones which lacked a simple, yet solid, motivation. \\

\end{abstract}

\begin{keywords}
  Smooth and Strongly Convex Optimization, Full Gradient Descent, Accelerated Gradient Descent, Heavy Ball method
\end{keywords}

\section{Introduction}

In the field of mathematical optimization one is interested in efficiently solving a minimization problem of the form
\begin{align} \label{opt:gen_min}
\min_{\bx\in X} f(\bx)
\end{align}
where the \emph{objective function}  $f$ is some real-valued function defined over the \emph{constraints set} $X$. Many core problems in the field of Computer Science, Economic, and Operations Research can be readily expressed in this form, rendering this minimization problem far-reaching. That being said, in its full generality this problem is just too hard to solve or even to approximate. As a consequence,  various structural assumptions on the objective function and the constraints set, along with better-suited optimization algorithms, have been proposed so as to make this problem viable.\\

One such case is smooth and strongly convex functions over some $d$-dimensional Euclidean space\footnote{More generally, one may consider smooth and strongly convex functions over some Hilbert space.}. Precisely, we consider continuously differentiable $f:\reals^d\to\reals$ which are \emph{$L$-smooth}, i.e.,
\begin{align*}
	\norm{\nabla f(\bx)- \nabla f(\by)} &\le L\norm{\bx-\by},\quad\forall \bx,\by\in\reals^d
\end{align*}
and \emph{$\mu$-strongly convex}, that is,
\begin{align*}
		f(\by) \ge f(\bx) + \inprod{\by-\bx}{\nabla f(\bx)} + \frac{\mu}{2}\normsq{\by-\bx},\quad\forall \bx,\by\in\reals^d
\end{align*}
A wide range of applications together with efficient solvers have made this family of problems very important. Naturally, an interesting question arises: how fast can these kind of problems be solved? better said, what is the computational complexity of minimizing smooth and strongly-convex functions to a given degree of accuracy?\footnote{Natural as these questions might look today, matters were quite different only few decades ago. In his book 'Introduction to Optimization' which dates back to 87', Polyak B.T devotes a whole section as to: 'Why Are Convergence Theorems Necessary?'  (See section 1.6.2 in \cite{polyak1987introduction}).}
Prior to answering these, otherwise ill-defined, questions, one must first address the exact nature of the underlying computational model. \\

Although being a widely accepted computational model in the theoretical computer sciences, the Turing Machine Model presents many obstacles when analyzing optimization algorithms. In their seminal work, \cite{nemirovskyproblem} evaded some of these difficulties by proposing the \emph{black box computational model}, according to which information regarding the objective function is acquired iteratively by querying an \emph{oracle}. This model does not impose any computational resource constraints\footnote{In a sense, this model is dual to the Turing Machine model where all the information regarding the parameters of the problem is available prior to the execution of the algorithm, but the computational resources are limited in time and space.}. Nemirovsky and Yudin showed that for any optimization algorithm which employs a first-order oracle, i.e. receives $(f(\bx),\nabla f(\bx))$ upon querying at a point $\bx\in\reals^d$, there exists an $L$-smooth $\mu$-strongly convex function $f:\reals^d\	\to\reals$, such that for any $\epsilon>0$ the number of oracle calls needed for obtaining an \emph{$\epsilon$-optimal} solution $\tilde{\bx}$, i.e.,
\begin{align} \label{ineq:eps_subopt}
f(\tilde{\bx}) < \min_{\bx\in\reals^d} f(\bx)+ \epsilon
\end{align}
must satisfy
\begin{align} \label{ineq:sqrtlb}
	 \# \text{ Oracle Calls} \ge \tilde{\Omega}\circpar{\min\left\{d, \sqrt{\kappa}\ln(1/\epsilon  \right\}}
\end{align}
where $\kappa\eqdef L/\mu$ denotes the so-called \emph{condition number}. \\

The result of Nemirovsky and Yudin can be seen as the starting point of the
present paper. The restricted validity of this lower bound to the first
$\cO\!\circpar{d}$ iterations is not a mere artifact of the analysis. Indeed,
from an information point of view, a minimizer of any convex quadratic
function can be found using no more than $\cO(d)$ first-order queries.
Noticing that this bound is attained by the Conjugate Gradient Descent method
(CGD, see \cite{polyak1987introduction}), it seems that one cannot get a
non-trivial lower bound once the number of queries exceeds the dimension $d$.
Moreover, a similar situation can be shown to occur for more general classes
of convex functions. However, the known algorithms which attain such behavior
(such as CGD and the center-of-gravity method, e.g., \cite{nemirovski2005efficient}) require
computationally intensive iterations, and are quite different than many
common algorithms used for large-scale optimization problems, such as
gradient descent and its variants. Thus, to capture the attainable
performance of such algorithms, we must make additional assumptions on their
structure. This can be made more solid using the following simple observation. \\

\emph{When applied on quadratic functions, the update rule of many optimization algorithms reduces to a recursive  application of a linear transformation which depends, possibly randomly, on the previous $p$ query points}.
\\
\\
Indeed, the update rule of CGD for quadratic functions is \emph{non-stationary}, i.e. uses a different linear transformation at each iteration, as opposed to other optimization algorithms which utilize less complex update rules such as: stationary updates rule, e.g., Gradient Descent, Accelerated Gradient Descent, Newton's method (see \cite{nesterov2004introductory}), The Heavy Ball method \cite{polyak1987introduction}, SDCA (see \cite{shalev2013stochastic}) and SAG (see \cite{roux2012stochastic}); cyclic update rules, e.g,. SVRG (see \cite{johnson2013accelerating}); and piecewise-stationary update rules, e.g., proximal methods and Accelerated SDCA (see \cite{shalev2013accelerated}). Inspired by this observation, in the present work we explore the boundaries of optimization algorithms which admit stationary update rules. We call such algorithms $p$-Stationary Canonical Linear Iterative optimization algorithms (abbr. $p$-SCLI), where $p$ designates the number of previous points which are necessary to generate new points. The quantity $p$ can be instructively interpreted as a limit on the amount of memory at the algorithm's disposal. \\

Similar to the analysis of power iteration methods, the convergence properties of such algorithms are intimately related to the eigenvalues of the corresponding linear transformation. Specifically, as the convergence rate of the recursive application of a linear transformation is essentially characterized by its largest magnitude eigenvalue, the asymptotic convergence rate of $p$-SCLI algorithms can be bounded from above and from below by analyzing the spectrum of the corresponding linear transformation. At this point we would like to remark that the technique of linearizing iterative procedures and analyzing their convergence behavior accordingly, which dates back to the pioneering work of the Russian mathematician Lyapunov, has been successfully applied in the field of mathematical optimization many times, e.g., \cite{polyak1987introduction} and more recently \cite{lessard2014analysis}. However, whereas previous works were primarily concerned with deriving upper bounds on the magnitude of the corresponding eigenvalues, in this work our reference point is lower bounds. \\

As eigenvalues are merely roots of characteristic polynomials\footnote{In fact, we will use a polynomial matrix analogous of characteristic polynomials which will turns out to be more useful for our purposes.}, our approach involves establishing a lower bound on the maximal modulus (absolute value) of the roots of polynomials. Clearly, in order to find a meaningful lower bound, one must first find a condition which is satisfied by all characteristic polynomials that correspond to $p$-SCLIs. We show that such condition does exist by proving that characteristic polynomials of consistent $p$-SCLIs, which correctly minimize the function at hand, must have a specific evaluation at $\lambda=1$. This in turn allows us to analyze the convergence rate purely in terms of the analytic theory of polynomials, i.e.,
\begin{align} \label{opt:intro_poly_lb}
	\textbf{Find} \quad\min\myset{\rho(q(z))}{q(z) \text{ is a real monic polynomial of degree } p \text{ and } q(1)=r    }  
\end{align}
where $r\in\reals$ and $\rho(q(z))$ denotes the maximum modulus over all roots of $q(z)$. Although a vast range of techniques have been developed for bounding the moduli of roots of polynomials (e.g. \cite{marden1966geometry,rahman2002analytic,milovanovic1994topics,walsh1922location,milovanovic2000distribution,fell1980zeros}), to the best of our knowledge, few of them address lower bounds (see \cite{higham2003bounds}. The minimization problem (\ref{opt:intro_poly_lb}) is also strongly connected with the question of bounding the spectral radius of 'generalized' companion matrices from below. Unfortunately, this topic too lacks an adequate coverage in the literature (see \cite{wolkowicz1980bounds,zhong2008bounds,horne1997lower,huang2007improving}). Consequently, we devote part of this work to establish new tools for tackling (\ref{opt:intro_poly_lb}). It is noteworthy that these tools are developed by using elementary arguments. This sharply contrasts with previously proof techniques used for deriving lower bounds on the convergence rate of optimization algorithms which employed heavy machinery from the field of extremal polynomials, such as Chebyshev polynomials (e.g., \cite{mason2002chebyshev}).\\

Based on the technique described above we present a novel lower bound on the convergence rate of $p$-SCLI  optimization algorithms. More formally, we prove that any $p$-SCLI optimization algorithm over $\reals^d$, whose iterations can be executed efficiently, requires 
\begin{align} \label{opt:lblb_conv_pcli}
\# \text{Oracle Calls} \ge \tilde{\Omega}\circpar{ \sqrt[p]{\kappa} \ln(1/\epsilon)}
\end{align}
in order to obtain an $\epsilon$-optimal solution, \emph{regardless of the dimension of the problem}. This result partially complements the lower bound presented earlier in \ineqref{ineq:sqrtlb}. More specifically, for $p=1$, we show that the runtime of algorithms whose update rules do not depend on previous points (e.g. Gradient Descent) and can be computed efficiently scales linearly with the condition number. For $p=2$, we get the optimal result for smooth and strongly convex functions. For $p>2$, this lower bound is clearly weaker than the lower bound shown in (\ref{ineq:sqrtlb}) at the first $d$ iterations. However, we show that it can be indeed attained by $p$-SCLI schemes, and surprisingly, some of them can be executed efficiently for certain classes of quadratic functions. Finally, we believe that a more refined analysis of problem (\ref{opt:intro_poly_lb}) would show that this technique is powerful enough to meet the classical lower bound $\sqrt{\kappa}$ for any $p$, in the worst-case over all quadratic problems.\\

The last part of this work concerns a cornerstone in the field of mathematical optimization, i.e., Nesterov's well-known Accelerated Gradient Descent method (AGD). At the time the work of Nemirovsky and Yudin was published, it was known that Gradient Descent (GD) obtains an $\epsilon$-optimal solution by issuing no more than 
\begin{align*}
\bigO{\kappa \ln (1/\epsilon)}
\end{align*}
first-order queries. The gap between this upper bound and the lower bound shown in (\ref{ineq:sqrtlb}) has intrigued many researchers in the field. Eventually, it was this line of inquiry that  led to the discovery of AGD by Nesterov (see \cite{nesterov1983method}), a slight modification of the standard GD algorithm, whose iteration  complexity is 
\begin{align*}
\bigO{\sqrt{\kappa} \ln (1/\epsilon)}
\end{align*}
Unfortunately, AGD lacks the strong geometrical intuition which accompanies many optimization algorithms, such as FGD and the Heavy Ball method. Primarily based on sophisticated algebraic manipulations, its proof strives for a more intuitive derivation (e.g. \cite{beck2009fast,baes2009estimate,tseng2008accelerated,sutskever2013importance,allen2014novel}). This downside has rendered the generalization of AGD to different optimization scenarios, such as constrained optimization problems, a highly non-trivial task which up to the present time does not admit a complete satisfactory solution. Surprisingly enough, by designing optimization algorithms whose characteristic polynomials are optimal with respect to a constrained version of (\ref{opt:intro_poly_lb}), we have uncovered a novel simple derivation of AGD. This reformulation as an optimal solution for a constrained optimization problem over polynomials, shows that AGD and the Heavy Ball are essentially two sides of the same coin.  \\
\\
To summarize, our main contributions, in order of appearance, are the following: 
\begin{itemize}
	\item We define a class of algorithms ($p$-SCLI) in terms of linear operations on the last $p$ iterations, and show that they subsume some of the most interesting algorithms used in practice.
	\item We prove that any $p$-SCLI optimization algorithm must use at least
	\begin{align*}
		\tilde{\Omega}\circpar{\sqrt[p]{\kappa}\ln(1/\epsilon)} 
	\end{align*}
	iterations in order to obtain an $\epsilon$-optimal solution. As mentioned earlier, unlike existing lower bounds, our bound holds for every fixed dimensionality.
	\item We show that there exist matching $p$-SCLI optimization algorithms which attain the convergence rates  stated above for all $p$. Alas, for $p\ge3$, an expensive pre-calculation task renders these algorithms inefficient. 
		
	\item As a result, we focus on a restricted subclass of $p$-SCLI optimization algorithms which can be executed efficiently. This yields a novel systematic derivation of Full Gradient Descent, Accelerated Gradient Descent, The Heavy-Ball method (and potentially others efficient optimization algorithms), each of which corresponds to an optimal solution of optimization problems on the moduli of polynomials' roots.
		
	\item We present new schemes which offer better utilization of second-order information by exploiting breaches in existing lower bounds. This leads to a new optimization algorithm which obtains a rate of $\sqrt[3]{\kappa}\ln(1/\epsilon)$
	in the presence of large enough spectral gaps.

\end{itemize}

\subsection{Notation} \label{subsec:notation}
We denote scalars with lower case letters and vectors with bold face letters. We use $\reals^{++}$ to denote the set of all positive real numbers. All functions in this paper are defined over Euclidean spaces equipped with the standard Euclidean norm and all matrix-norms are assumed to denote the spectral norm. \\

We denote a block-diagonal matrix whose blocks are $A_1,\dots,A_k$ by the conventional direct sum symbol, i.e., $\oplus_{i=1}^k A_k$. We devote a special operator symbol for scalar matrices $\text{Diag}\circpar{a_1,\dots,a_d} = \oplus_{i=1}^d a_i$. The spectrum of a square matrix $A$ and its spectral radius, the maximum magnitude over its eigenvalues, are denoted by $\spec{A}$ and $\rho(A)$, respectively. Recall that the eigenvalues  of a square matrix $A\in\reals^{d\times d}$ are exactly the roots of the characteristic polynomial which is defined as follows
\begin{align*}
	\chi_A(\lambda) &= \det(A-\lambda I_d) 
\end{align*}
where $I_d$ denotes the identity matrix. Since polynomials in this paper have their origins as characteristic polynomials of some square matrices, by a slight abuse of notation, we will denote the roots of a polynomial $q(z)$ and its root radius, the maximum modulus over its roots, by $\spec{q(z)}$ and $\rho(q(z))$, respectively, as well.
\\
\\
The following notation for quadratic functions and matrices will be of frequent use,  
\begin{align*}
		\posdef{d}{\Sigma} &\eqdef \left\{A\in\reals^{d\times d}\middle| A \text{ is symmetric and } \spec{A} \subseteq \Sigma\right\}\\
		\posdefun{d}{\Sigma} &\eqdef \myset{\quadab{A,\bb}}{A\in \posdef{d}{\Sigma},\bb\in\reals^d}
\end{align*}
where $\Sigma$ denotes a non-empty set of positive reals, and where $\quadab{A,\bb}$ denotes the following quadratic function
\begin{align*}
	\quadab{A,\bb} = \bx^\top A \bx + \bb^\top \bx 
\end{align*}

\section{Framework}
In the sequel we establish our framework for analyzing optimization algorithms for minimizing smooth and strongly convex functions. First, to motivate this technique, we show that the analysis of SDCA presented in \cite{shalev2013stochastic} is tight by using a similar method. Next, we lay the foundations of the framework by generalizing and formalizing various aspects of the SDCA case. We then examine some popular optimization algorithms through this formulation. Apart from setting the boundaries for this work, this inspection gives rise to, otherwise subtle, distinctions between different optimization algorithms. Lastly, we discuss the computational complexity of $p$-SCLIs, as well as their convergence properties.

\subsection{Case Study - Stochastic Dual Coordinate Ascent} \label{section:sdca_case_study}
We consider an optimization algorithm called Stochastic Dual Coordinates Ascent (SDCA\footnote{For a detailed analysis of SDCA, please refer to \cite{shalev2013stochastic}.}) for solving Regularized Loss Minimization (RLM) problems (\ref{opt:RLM}), which are of great significance for the field of Machine Learning. It is shown that applying SDCA on quadratic loss functions allows one to reformulate it as a recursive application of linear transformations. The relative simplicity of such processes is then exploited to derive a lower bound on the convergence rate.\\

A smooth-RLM problem is an optimization task of the following form 
\begin{align} \label{opt:RLM}
\min_{\bw\in\reals^d}P(\bw) &\eqdef \frac{1}{n} \sum_{i=1}^n \phi_i(\bw^\top \bx_i) + \frac{\lambda}{2} \normsq{\bw}
\end{align}
where $\phi_i$ are $1/\gamma$-smooth and convex,  $\bx_1,\ldots,\bx_n$ are vectors in $\reals^d$ and $\lambda$ is a positive constant. For ease of presentation, we further assume that $\phi_i$ are non-negative, $\phi_i(0)\le 1$ and $\norm{\bx_i}\le 1$ for all $i$. \\

The optimization algorithm SDCA works by minimizing an equivalent optimization problem 
\begin{align*}
\min_{\alpha\in\reals^n} D(\balpha) \eqdef \frac{1}{n}\sum_{i=1}^n \phi_i^\star(\alpha_i) +\frac{1}{2\lambda n^2} \normsq{\sum_{i=1}^n \alpha_i \bx_i}
\end{align*}
where $\phi^\star$ denotes the Fenchel conjugate of $\phi$, by repeatedly picking $z\sim \cU([n])$ uniformly and minimizing $D(\balpha)$ over the $z$'th coordinate. The latter optimization problem is referred to as the \emph{dual problem}, while the problem presented in (\ref{opt:RLM}) is called the \emph{primal problem}.
As shown in \cite{shalev2013stochastic}, it is possible to convert a high quality solution of the dual problem into a high quality solution of the primal problem. This allows one to bound from above the number of iterations required for obtaining a prescribed level of accuracy $\epsilon>0$ by 
\begin{align*} \label{bigo:RLM_conv_rate}
\bigtO{\circpar{n+\frac{1}{\lambda \gamma}}\ln(1/\epsilon)}\\
\end{align*}

Let us show that this analysis is indeed tight. First, let us define the following $2$-smooth functions
\begin{align*}
\phi_i(y) = y^2,\quad i=1,\dots,n
\end{align*}
and let us define $\bx_1=\bx_2=\cdots=\bx_n=\frac{1}{\sqrt{n}}\mathbbm{1}$. This yields
\begin{align} 
D(\balpha) &= \frac{1}{2}\balpha^\top \circpar{ \frac{1}{2n}I+ \frac{1}{\lambda n^2} \mathbbm{1}\mathbbm{1}^\top}\balpha
\end{align}
Clearly, the unique minimizer of $D(\balpha)$ is $\balpha^*\eqdef0$. Now, given $i\in[n]$ and $\balpha\in\reals^n$ , it is easy to verify that 
\begin{align}
\argmin_{\alpha' \in\reals} D(\alpha_1,\dots,\alpha_{i-1},\alpha',\alpha_{i+1},\dots,\alpha_{n}) = \frac{-2}{2+\lambda n} \sum_{j\neq i} \alpha_j 
\end{align}
Thus, the next test point $\balpha^+$, generated by taking a step along the $i$'th coordinate, is linear transformation of the previous point, i.e.,
\begin{align} \label{eq:rlm_costep}
\balpha^+=\circpar{I-\be_i \bu_i^\top}\balpha
\end{align}
Where 
\begin{align*}
\bu_i^\top &\eqdef \circpar{\frac{2}{2+\lambda n }, \dots, \frac{2}{2+\lambda n },\underbrace{1}_{i\text{'s entry}},\frac{2}{2+\lambda n },\ldots, \frac{2}{2+\lambda n } }
\end{align*}
Let $\balpha^k,~ k=1,\dots,K$ denote the $k$'th test point. The sequence of points $(\balpha^k)_{k=1}^K$ is randomly generated by minimizing $D(\balpha)$ over the $z_i$'th coordinate at the $i$'th iteration, where $z_1,z_2,\dots,z_K\sim \mathcal{U}([n])$ is a sequence of $K$ uniform distributed i.i.d random variables. Applying (\ref{eq:rlm_costep}) over and over again starting from some initialization point $\balpha^0$ we obtain 
\begin{align*}
\balpha^k&=\circpar{I-\be_{z_K}\bu_{z_K}^\top}\circpar{I-\be_{z_{K-1}}\bu_{z_{K-1}}^\top}\cdots\circpar{I-\be_{z_1}\bu_{z_1}^\top}\balpha^0
\end{align*}
To compute $\mathbb{E}[\balpha^K]$ note that by the i.i.d hypothesis and by the linearity of the expectation operator,
\begin{align}
\mathbb{E}\left[\balpha^K\right]&=\mathbb{E}\left[\circpar{I-\be_{z_K}\bu_{z_K}^\top}\circpar{I-\be_{z_{K-1}}\bu_{z_{K-1}}^\top}\cdots\circpar{I-\be_{z_1}\bu_{z_1}^\top}\balpha^0\right]\nonumber\\
&=\mathbb{E}\left[\circpar{I-\be_{z_K}\bu_{z_K}^\top}\right]\mathbb{E}\left[\circpar{I-\be_{z_{K-1}}\bu_{z_{K-1}}^\top}\right]\cdots\mathbb{E}\left[\circpar{I-\be_{z_1}\bu_{z_1}^\top}\right]\balpha^0\nonumber\\
&=\mathbb{E}\left[\circpar{I-\be_{z}\bu_{z}^\top}\right]^K\balpha^0 \label{kpoint}
\end{align}
The convergence rate of this process is governed by the spectral radius of $$E\eqdef \mathbb{E}\left[I-\be_{z}\bu_{z}^\top\right]$$A straightforward calculation shows that the eigenvalues of $E$, ordered by magnitude, are
\begin{align}
\underbrace{\frac{1}{2/\lambda + n},\dots , \frac{1 }{2/\lambda + n}}_{n-1 \text{ times}}, 1 - \frac{2 + \lambda}{2+\lambda n} 
 \end{align}
By choosing $\balpha^0$ to be the following normalized eigenvector which corresponds to the largest eigenvalue,
$$\balpha^0=\circpar{\frac{1}{\sqrt{2}},-\frac{1}{\sqrt{2}},0,\dots,0}$$
and plugging it into \eqref{kpoint}, we can now bound from below the distance of $\mathbb{E}[\balpha^K]$ to the optimal point $\balpha^*=0$,
\begin{align*}
\norm{\mathbb{E}\left[\balpha^K\right]-\balpha^*}
&=\norm{\mathbb{E}\left[\circpar{I-\be_{z}\bu_{z}^\top}\right]^K\balpha^0}\\
&=\circpar{1 - \frac{1}{2/\lambda+n}}^K\norm{\balpha^0}\\
&=\circpar{1 - \frac{2}{(4/\lambda+2n-1)+1}}^K\\
&\ge \circpar{\exp\circpar{\frac{-1}{2/\lambda+n-1}}}^K
\end{align*}
Where the last inequality is due to the following inequality,
\begin{align} \label{ineq:exp_x}
1-\frac{2}{x+1}\ge \exp\circpar{\frac{-2}{x-1}},\quad \forall x\ge1
\end{align}
We see that the minimal number of iterations required for obtaining a solution whose distance form the $\balpha^*$ is less than $\epsilon>0$ must be greater than 
\begin{align*} 
 \circpar{2/\lambda+n-1}\ln\circpar{1/\epsilon} 
\end{align*}
Thus showing that, up to logarithmic factors, the analysis of the convergence rate of SDCA is tight.

\subsection{Definitions}
In the sequel we introduce the framework of $p$-SCLI optimization algorithms which generalizes the analysis shown in the preceding section.\\

We denote the set of $d\times d$ symmetric matrices whose spectrum lies in $\Sigma\subseteq\reals^{++}$ by $\posdef{d}{\Sigma}$ and denote the following set of quadratic functions 
\begin{align*}
	f_{A,\bb}(\bx) \eqdef\frac{1}{2}\bx^\top A\bx + \bb^\top \bx,\quad A\in\posdef{d}{\Sigma}
\end{align*}
by $\posdefun{d}{\Sigma}$. Note that since twice continuous differentiable functions $f(\bx)$ are $L$-smooth and $\mu$-strongly convex if and only if $$\spec{\nabla^2 (f(\bx))}\subseteq [\mu,L]\subseteq\reals^{++},\quad \bx\in\reals^d$$ we have that $\posdefun{d}{[\mu,L]}$ comprises $L$-smooth $\mu$-strongly convex quadratic functions. Thus, any optimization algorithm designed for minimizing smooth and strongly convex functions can be used to minimize functions in $\posdefun{d}{[\mu,L]}$. The key observation here is that since the gradient of $\quadab{A,\bb}$ is linear in $\bx$, when applied to quadratic functions, the update rules of many optimization algorithms also become linear in $\bx$. This formalizes as follows.
\begin{definition} [$p$-SCLI optimization algorithms] \label{definition:pscli}
An optimization algorithm $\cA$ is called a $p$-stationary canonical linear iterative (abbr. $p$-SCLI) optimization algorithm over $\reals^d$ if there exist $p+1$ mappings $C_0(X),C_1(X),\dots,C_{p-1}(X),N(X)$ from $\reals^{d\times d}$ to $\reals^{d\times d}$-valued random variables, such that for any $\quadab{A,\bb}\in\posdefun{d}{\Sigma}$ the corresponding initialization and update rules take the following form:
\begin{align} 
&\bx^0,\bx^1,\dots,\bx^{p-1}\in\reals^d \label{def:pscli_initialization_points}\\
&\bx^k = \sum_{j=0}^{p-1} C_{j}(A) \bx^{k-p+j}+ N(A)\bb,\quad k=p,p+1,\dots \label{def:pscli_update_rule}
\end{align}
We further assume that in each iteration $C_j(A)$ and $N(A)$ are drawn independently of previous realizations\footnote{In this context, this assumption is usually referred to as stationarity.}, and that $\bE C_i(A)$ are finite and simultaneously triangularizable\footnote{Intuitively, having this technical requirement is somewhat similar to assuming that the coefficients matrices commute (see \cite{drazin1951some} for a precise statement), and as such does not seem to restrict the scope of this work. Indeed, it is common to have $\bE C_i(A)$ as polynomials in $A$ or as diagonal matrices, in which case the assumption holds true.}.
\end{definition}

Let us introduce a few more definitions and terminology which will be used throughout this paper. The number of previous points $p$ by which new points are generated is called the \emph{lifting factor}. The matrix-valued random variables $C_0(X),C_1(X),\dots,C_{p-1}(X)$ and $N(X)$ are called \emph{coefficient matrices} and \emph{inversion matrix}, respectively. The term inversion matrix refers to the mapping $N(X)$, as well as to a concrete evaluation of it. It will be clear from the context which interpretation is being used. The same comment holds for coefficient matrices. \\

As demonstrated by the following definition, coefficients matrices of $p$-SCLIs can be equivalently described in terms of polynomial matrices\footnote{For a detailed cover of polynomial matrices see \cite{gohberg2009matrix}.}. This correspondence will soon play a pivotal role in the analysis of $p$-SCLIs.
\begin{definition} \label{def:l_lambda}
The characteristic polynomial of a given $p$-SCLI optimization algorithm $\cA$ is defined by 
\begin{align} 
	\syspol{\cA}{\lambda,X} &\eqdef  I_d\lambda^p - \sum_{j=0}^{p-1} \bE C_j(X) \lambda^{j}  
\end{align}
where $C_j(X)$ denote the coefficient matrices. Moreover, given $X\in\reals^{d\times d}$ we define the root radius of $\syspol{\cA}{\lambda,X}$ by
\begin{align*}
	\rho_\lambda(\syspola)&= \rho(\det{\cL(\lambda,X)}) = \max\myset{|\lambda'|}{\det{\cL(\lambda^\prime,X)}=0} 
\end{align*}
\end{definition}
For the sake of brevity, we will sometimes specify a given $p$-SCLI optimization algorithm $\cA$ using an ordered pair of a characteristic polynomial and an inversion matrix as follows $$\cA\eqdef(\syspol{\cA}{\lambda,X},N(X))$$ 

Lastly, note that nowhere in the definition of $p$-SCLIs did we assume that the optimization process converges to the minimizer of the function under consideration - an assumption which we refer to as \emph{consistency}. 
\begin{definition} [Consistency of $p$-SCLI optimization algorithms] \label{definition:consistency}
A $p$-SCLI optimization algorithm $\cA$ is said to be consistent with respect to a given $A\in\posdef{d}{\Sigma}$ if for any $\bb\in\reals^d$, $\cA$ converges to the minimizer of $\quadab{A,\bb}$, regardless of the initialization point. That is, for $\left(\bx^k\right)_{k=1}^\infty$ as defined in (\ref{def:pscli_initialization_points},\ref{def:pscli_update_rule}) we have that 
\begin{align*}
	\bx^k\to-A^{-1}\bb
\end{align*} 
for any $\bb\in\reals^d$. Furthermore, if $\cA$ is consistent with respect to all $A\in\posdef{d}{\Sigma}$, then we say that $\cA$ is consistent with respect to $\posdefun{d}{\Sigma}$.
\end{definition}

\subsection{Specifications for Some Popular optimization algorithms} \label{section_spec_algo}
Having defined the framework of $p$-SCLI optimization algorithms, a natural question now arises: how broad is the scope of this framework and what does characterize 
optimization algorithms which it applies to? Loosely speaking, any optimization algorithm whose update rules depend linearly on the first and the second order derivatives of the function under consideration is eligible for this framework. Instead of providing a precise characterization for such algorithms, we apply various popular optimization algorithms on a general quadratic function $\quadab{A,\bb}\in\posdefun{d}{[\mu,L]}$ and then re-express them as $p$-SCLI optimization algorithms.\\

\begin{description}
\item [Full Gradient Descent (FGD)  ] \label{spec:fgd} is a $1$-SCLI optimization algorithm,
			\begin{align*}
			\bx^0 &\in \reals^d\\
			\bx^{k+1} &= \bx^k - \beta \nabla f(\bx^k)= \bx^k - \beta(A\bx^k +\bb)=(I-\beta A)\bx^k -\beta \bb\\
			\beta &= \frac{2}{\mu +L}
			\end{align*}
			See \cite{nesterov2004introductory} for more details.
\item [Newton method] \label{spec:newton} is a $0$-SCLI optimization algorithm.
\begin{align*}
\bx^0 &\in \reals^d\\
\bx^{k+1} &=  \bx^k- (\nabla^{2} f(\bx^k))^{-1} \nabla f(\bx^k) = \bx^k- A^{-1}(A\bx^k + \bb)\\&= (I-A^{-1}A)\bx^k - A^{-1}\bb = -A^{-1}\bb
\end{align*}
Note that Newton method can be also formulated as a degenerate $p$-SCLI for some $p\in\bN$, whose coefficients matrices vanish. See \cite{nesterov2004introductory} for more details. 
\item [The Heavy Ball Method] is a $2$-SCLI optimization algorithm.
\begin{align*}
\bx^{k+1}  &= \bx^k - \alpha\nabla f(\bx^{k})+ \beta (\bx^{k}-\bx^{k-1}) \\ 
&= \bx^k - \alpha(A\bx^{k} + \bb)+ \beta (\bx^{k}-\bx^{k-1}) \\&= \circpar{(1+\beta) I-\alpha A } \bx^{k} -\beta I \bx^{k-1} -\alpha \bb\\
\alpha &= \frac{4}{\circpar{\sqrt{L}+\sqrt{\mu}}^2},\quad
\beta = \circpar{\frac{\sqrt{L}-\sqrt{\mu}}{\sqrt{L}+\sqrt{\mu}} }^2
\end{align*}
See \cite{polyak1987introduction} for more details.
\item [Accelerated Gradient Descent (AGD)]\label{spec:agd} is a $2$-SCLI optimization algorithm. 
\begin{align*}
\bx^0&=\by^0 \in \reals^d\\
\by^{k+1} &= \bx^{k} - \frac1L\nabla f(\bx^{k})\\
\bx^{k+1} &= \circpar{1+\alpha}\by^{k+1} - \alpha \by^{k} \\
\alpha &= \frac{\sqrt{L}-\sqrt{\mu}}{\sqrt{L}+\sqrt{\mu}}
\end{align*}
Which can be rewritten as,
\begin{align*}
\bx^0& \in \reals^d\\
\bx^{k+1} &= \circpar{1+\alpha}\circpar{\bx^{k} - \frac1L\nabla f(\bx^{k})} - \alpha \circpar{\bx^{k-1} - \frac1L\nabla f(\bx^{k-1})}\\
&= \circpar{1+\alpha}\circpar{\bx^{k} - \frac1L (A\bx^{k}+\bb)} - \alpha \circpar{\bx^{k-1} - \frac1L (A\bx^{k-1}+\bb)}\\
&= \circpar{1+\alpha}\circpar{I - \frac1L A} \bx^{k}
 -\alpha\circpar{I - \frac1L A} \bx^{k-1}
 -\frac1L \bb
\end{align*}
See \cite{nesterov2004introductory} for more details.

\item [Stochastic Coordinate Descent (SCD)] is a $1$-CLI optimization algorithm. This is an extension of the example shown in Section \ref{section:sdca_case_study}. SCD acts by repeatedly minimizing a uniformly randomly drawn coordinate in each iteration. That is,
\begin{align*}
&\bx^0 \in \reals^d\\
&\text{Pick } i\sim \cU([d]) \text{ and set }\bx^{k+1} =  \left(I-\frac{1}{A_{i,i}}\be_i \mathbf{a}_{i,\star}^\top \right)\bx^k - \frac{b_i}{A_{i,i}}\be_i
\end{align*}
where $\mathbf{a}_{i,\star}^\top$ denotes the $i$'th row of $A$ and $\bb\eqdef\circpar{b_1,b_2,\dots,b_d}$.  Note that the expected update rule of this method is equivalent to the well-known Jacobi's iterative method.
\end{description}

We now describe some popular optimization algorithms which do not fit this framework, mainly because the stationarity requirement fails to hold. The extension of this framework to cyclic and piecewise stationary optimization algorithms is left to future work.

\begin{description}
\item [Conjugate Gradient Descent (CGD)] Can be expressed as a non-stationary linear iterative method. 
\begin{align*}
\bx^{k+1}  &= \circpar{(1+\beta_k) I-\alpha_k A } \bx^{k} -\beta_k I \bx^{k-1} -\alpha_k \bb
\end{align*}
where $\alpha_k$ and $\beta_k$ are computed at each iteration based on $\bx^k,\bx^{k-1},A$ and $b$. Note the similarity of CGD and the heavy ball method. See \cite{polyak1987introduction,nemirovski2005efficient} for more details.

\item [Stochastic Gradient Descent (SGD) ] A straightforward extension of the deterministic FGD. Specifically,  let $(\Omega,\cF,\cP)$ be a probability space  and let $G(\bx,\omega):\reals^d\times \Omega\to\reals^d$ be 
an unbiased estimator of $\nabla f(\bx)$ for any $\bx$. That is,
\begin{align*}
\bE[ G(\bx,\omega) ] &= \nabla f(\bx) = A \bx + \bb,\quad \bx\in\reals^d
\end{align*}
Equivalently, define $\be(\bx,\omega)= G(\bx,\omega) - (A \bx + \bb) $ and assume $\bE[\be(\bx,\omega)]=0,~\bx\in\reals^d$. SGD may be defined using a suitable sequence of step sizes $(\gamma_i)_{i=1}^\infty$ as follows
\begin{align*}
			\text{Generate } \omega_{k} \text{ randomly and set } \bx^{k+1} &= \bx^k - \gamma_i G(\bx^k,\omega_{k})\\ &= \circpar{I-\gamma_i A}\bx^k  - \gamma_i \bb - \gamma_i \be(\bx,\omega)
\end{align*}
Clearly, some types of noise may not form a $p$-SCLI optimization algorithm. However, for some instances, e.g., quadratic learning problems, we have $$\be(\bx,\omega)=A_\omega\bx + \bb_\omega$$ such that 
\begin{align*}
\bE[A_\omega]&=0,\quad \bE[\bb_\omega]=0
\end{align*}
If, in addition, the step size is fixed then we get a $1$-SCLI optimization algorithm. See \cite{kushner2003stochastic,spall2005introduction,nemirovski2005efficient} for more details.

\end{description}


\subsection{Computational Complexity}
The stationarity property of general $p$-SCLIs optimization algorithms implies that the computational cost of minimizing a given quadratic function $\quadab{A,\bb}$, assuming $\Theta\circpar{1}$ cost for all arithmetic operations, is
\begin{align*}
	\# \text{ Iterations } \times
	\begin{cases} 
	\text{ Generating coefficient and inversion matrices randomly}\\
	\qquad\qquad+\\
	\text{ Executing update rule (\ref{def:pscli_update_rule})based on the previous } p \text{ points }
	\end{cases}
\end{align*}
The computational cost of the execution of update rule (\ref{def:pscli_update_rule}) scales linearly with $d$  the dimension of the problem and $p$ the lifting factor. Thus, the running time of $p$-SCLIs is mainly affected by the iterations number and the computational cost of randomly generating coefficient and inversion matrices each time. Notice that for deterministic $p$-SCLIs one can save running time by computing the coefficient and inversion matrices once, prior to the execution of the algorithm. Not surprisingly, but interesting nonetheless, there is a law of conservation which governs the total amount of computational cost invested in both factors: the more demanding is the task of randomly generating coefficient and inversion matrices, the less is the total number of iterations required for obtaining a given level of accuracy, and vice verse. Before we can make this statement more rigorous, we need to present a few more facts about $p$-SCLIs. For the time being, let us focus on the \emph{iteration~complexity}, i.e., the total number iterations, which forms our analogy for black box complexity. \\

The \emph{iteration~complexity} of a $p$-SCLI optimization algorithm $\cA$ with respect to an accuracy level $\epsilon$, an initialization points $\mathcal{X}^0$ and a quadratic function $\quadab{A,\bb}$, symbolized by
$$\IC_\cA\circpar{\epsilon,\quadab{A,\bb},\mathcal{X}^0}$$
is defined to be the minimal number of iterations $K$ such that 
\begin{align*} 
\norm{\bE [\bx^k - \bx^*] } <\epsilon,\quad \forall k\ge K
\end{align*}
where $\bx^*=-A^{-1}\bb$ is the minimizer of $\quadab{A,\bb}$, assuming $\cA$ is initialized at $\mathcal{X}^0$. We would like to point out that although iteration complexity is usually measured through 
\begin{align*} 
\bE\norm{ \bx^k - \bx^*} 
\end{align*}
here we employ a different definition. We will discuss this issue shortly. \\

In addition to showing that the iteration complexity of $p$-SCLI algorithms scales logarithmically with $1/\epsilon$, the following theorem provides a characterization for the iteration complexity in terms of the root radius of the characteristic polynomial. The full proof for this theorem is somewhat long and thus provided in Section \ref{subsection:conv_prop}.
\begin{theorem} \label{thm:ic_cli}
Let $\cA$ be a $p$-SCLI optimization algorithm over $\reals^d$ and let $\quadab{A,\bb}\in\posdefun{d}{\Sigma},~(\Sigma\subseteq \reals^{++})$ be a quadratic function. Then, there exists $\cX^0 \in \reals^{dp}$ such that
\begin{align*}
\IC_\cA\circpar{\epsilon,\quadab{A,\bb},\mathcal{X}^0}=\tilde{\Omega}\circpar{\frac{\rho}{1-\rho}\ln(1/\epsilon)}
\end{align*}
and for all $\cX^0 \in \reals^{dp}$, it holds that
\begin{align*}
\IC_\cA\circpar{\epsilon,\quadab{A,\bb},\mathcal{X}^0}=\bigtO{\frac{1}{1-\rho}\ln(1/\epsilon)}
\end{align*}
where $\rho$ denotes the root radius of the characteristic polynomial evaluated at $X=A$. 
\end{theorem}  
We remark that the constants in the asymptotic behavior above may depend on the quadratic function under consideration, and that the logarithmic terms depend on the distance of the initialization points from the minimizer, as well as the lifting factor and the spectrum of the second-order derivative. For the sake of clarity, we shall usually omit the dependency on the initialization points.\\

There are two, rather subtle, issues regarding the definition of iteration complexity which we would like to address. First, observe that in many cases a given point $\tilde{\bx}\in\reals^d$ is said to be $\epsilon$-optimal w.r.t  some real function $f:\reals^d\to\reals$ if 
\begin{align*}
	f(\tilde{\bx}) < \min_{\bx\in\reals^d} f(\bx) + \epsilon 
\end{align*}
However, here we employ a different measure for optimality. Fortunately, in our case either can be used without essentially affecting the iteration complexity. That is, although in general the gap between these two definitions can be made arbitrarily large, for $L$-smooth $\mu$-strongly convex functions we have
\begin{align*}
	\frac{\mu}{2} \normsq{\bx-\bx^*} \le  f(\bx)-f(\bx^*) \le \frac{L}{2} \normsq{\bx-\bx^*} 
\end{align*} 
Combining these two inequalities with the fact that the iteration complexity of $p$-SCLIs depends logarithmically on $1/\epsilon$ implies that in this very setting these two distances are interchangeable, up to logarithmic factors.\\

Secondly, here we measure the sub-optimality of the $k$'th iteration by $\norm{\bE [\bx^k - \bx^*] }$,
whereas in many other stochastic settings it is common to derive upper and lower bounds on $\bE\rectpar{\norm{\bx^k - \bx^* }}$. That being the case, by
\begin{align*}
	\bE\rectpar{\normsq{\bx^k - \bx^* }} 	&= \bE\rectpar{\normsq{\bx^k - \bE \bx^k}} +\normsq{ \bE \rectpar{\bx^k  -\bx^* }} 
\end{align*}
we see that if the variance of the $k$'th point is of the same order of magnitude as the norm of the expected distance from the optimal point, then both measures are equivalent. Consequently, our upper bounds imply upper bounds on $\bE\rectpar{\normsq{\bx^k - \bx^* }}$ for deterministic algorithms (where the variance term is zero), and our lower bounds imply lower bounds on $\bE\rectpar{\normsq{\bx^k - \bx^* }}$, for both deterministic and stochastic algorithms (since the variance is always non-negative). We defer a more adequate treatment for this matter to future work.

\section{Deriving Bounds for \texorpdfstring{$p$}{p}-SCLI Algorithms}
The goal of the following section is to show how the framework of $p$-SCLI optimization algorithms can be used to derive lower and upper bounds. Our presentation follows from the simplest setting to the most general one. case to th First, we present a useful characterization of consistency (see \defref{definition:consistency}) of $p$-SCLIs using the characteristic polynomial. Next, we demonstrate the importance of consistency through a simplified one dimensional case. This line of argument is then generalized to finite dimensional spaces and is used to explain the role of the inversion matrix. Finally, we conclude this section by providing a schematic description of this technique for the most general case which is used both in Section (\ref{chapter:lower_bounds}) to establish lower bounds on the convergence rate of $p$-SCLIs with diagonal inversion matrices, and in Section (\ref{section:ub}) to derive efficient $p$-SCLIs.\\

\subsection{Consistency} \label{subsection:cons}
Closely inspecting various specifications for $p$-SCLI optimization algorithms (see Section (\ref{section_spec_algo})) reveals that the coefficient matrices always sum up to $I+\bE N(X)X$, where $N(X)$ denotes the inversion matrix. It turns out that this is not a mere coincidence, but an extremely useful characterization for consistency of $p$-SCLIs. To see why this condition must hold, suppose $\cA$ is a deterministic  $p$-SCLI algorithm over $\reals^d$ whose coefficient matrices and inversion matrix are $C_0(X),\dots,C_{p-1}(X)$ and  $N(X)$, respectively, and suppose that $\cA$ is consistent w.r.t some $A\in\posdef{d}{\Sigma}$. Recall that every $p+1$ consecutive points generated by $\cA$ are related by (\ref{def:pscli_update_rule}) as follows 
\begin{align*} 
\bx^k = \sum_{j=0}^{p-1} C_{j}(A) \bx^{k-p+j}+ N(A)\bb,\quad k=p,p+1,\dots 
\end{align*}
Taking limit of both sides of the equation above and noting that by consistency
\begin{align*}
	\bx^k &\to -A^{-1}\bb 
\end{align*}
for any $\bb\in\reals^d$, yields
\begin{align*} 
-A^{-1}\bb  = -\sum_{j=0}^{p-1} C_{j}(A) A^{-1}\bb  + N(A)\bb
\end{align*}
Thus,
\begin{align*} 
-A^{-1}  = -\sum_{j=0}^{p-1} C_{j}(A) A^{-1}  + N(A)
\end{align*}
Multiplying by $A$ and rearranging, we obtain
\begin{align} \label{equation:sum_for_consistency} 
\sum_{j=0}^{p-1} C_{j}(A) = I_d  + N(A)A
\end{align}
On the other hand, if instead of assuming consistency we assume that $\cA$ generates a convergent sequence of points and that \eqref{equation:sum_for_consistency} holds, then the arguments used above show that the limit point must be $-A^{-1}\bb$. In terms of the characteristic polynomial of $p$-SCLIs, this formalized as follows.
\begin{theorem} [Consistency - System Polynomials]\label{thm:conv_correct}
Suppose $\cA\eqdef(\syspola,N(X))$ is a $p$-SCLI optimization algorithm. Then, $\cA$ is consistent with respect to $A\in\posdef{d}{\Sigma}$ if and only if the following two conditions hold:
\begin{align}
1.&~\syspol{\cA}{1,A} = -\bE N(A)A \label{consis_1_syspol}\\
2.&~\rho_\lambda(\syspol{}{\lambda,A}) < 1 \label{consis_2_syspol}
\end{align}
\end{theorem}
The proof for the preceding theorem is provided in Section \ref{section:conv_correct}.
This result will be used extensively throughout the reminder of this work.

\subsection{Simplified One-Dimensional Case} \label{section:odc}
To illustrate the significance of consistency in the framework of $p$-SCLIs, consider the following simplified case. Suppose $\cA$ is a deterministic 2-SCLI optimization algorithm over $\posdefun{1}{[\mu,L]}$, such that its inversion matrix $N(x)$ is some constant scalar $\nu\in\reals$ and its coefficient matrices $c_0(x),c_1(x)$ are free to take any form. The corresponding characteristic polynomial is
\begin{align*}
	\mathcal{L}(\lambda,x)  &= \lambda^2 - c_1(x)\lambda -c_0(x)
\end{align*}
Now, let $f_{a,b}(x)\in\posdefun{1}{[\mu,L]}$ be a quadratic function. By \thmref{thm:ic_cli}, we know that $\cA$ converges to the minimizer of $f_{a,b}(x)$ with an asymptotic geometric rate of $\rho_\lambda(\mathcal{L}(\lambda,a))$, the maximal modulus root. Thus, ideally we would like to set $c_j(x)=0,~j=0,1$. However, this might violate the consistency condition (\ref{consis_1_syspol}), according to which, one must maintain $$\mathcal{L}(1,a)=-\nu a$$ That being the case, how little can $\rho_\lambda\circpar{\mathcal{L}(\lambda,a)}$ be over all possible choices for $c_j(a)$ which satisfy $\mathcal{L}(1,a)=-\nu a$? Formally, we seek to solve the following minimization problem. 
\begin{align*} 
\rho_* = \min\left\{ \rho_\lambda(\syspol{}{\lambda,a})~ \left|~ \syspol{}{\lambda,a} \text{ is a real monic quadratic polynomial in $\lambda$ and } \syspol{}{1} = -\nu a \right.\right\}
\end{align*}
By consistency we also have that $\rho_*$ must be strictly less than one. This readily implies that $-\nu a>0$. In which case, \lemref{lem:comp_poly} below gives
\begin{align}
	\rho_* & \ge \rho\circpar{\circpar{\lambda-1-\sqrt{-\nu a}}^2} = \absval{\!\sqrt{-\nu a}-1\!}
\end{align}
The key ingredient here is that $\nu$ cannot be chosen so as to be optimal for all $\posdefun{1}{[\mu,L]}$, at one and the same time. Indeed, the preceding inequality holds in particular for $a=\mu$ and $a=L$, by which we conclude that 
\begin{align} \label{ineq:one_dim_case}
	\rho_* &\ge \max\left\{\absval{\!\sqrt{-\nu \mu}-1\!}, \absval{\!\sqrt{-\nu L}-1\!} \right\} \ge \frac{\sqrt{\kappa}-1}{\sqrt{\kappa}+1}
\end{align}
where $\kappa\eqdef L/\mu$. Plugging in \ineqref{ineq:one_dim_case} into \thmref{thm:ic_cli} implies that there exists $f_{a,b}(x)\in\posdefun{1}{[\mu,L]}$ such that the iteration complexity of $\cA$ for minimizing it is 
\begin{align*}
	\tilde{\Omega}\circpar{\frac{\sqrt{\kappa} -1}{2}\ln(1/\epsilon)}  
\end{align*}
To conclude, by applying this rather natural line of argument we have established a lower bound on the convergence rate of any $2$-SCLI optimization algorithms for smooth and strongly convex function over $\reals$, e.g., AGD and HB. 

\subsection{General Case and the Role of the Inversion Matrix} \label{subsection:general_case}
We now generalize the analysis shown in the previous simplified case to any $p$-SCLI optimization algorithm over any finite dimensional space. This generalization relies on a useful decomposability property of the characteristic polynomial, according to which deriving a lower bound on the convergence rate of $p$-SCLIs over $\reals^d$ is essentially equivalent for deriving $d$ lower bounds on
the maximal modulus of the roots of $d$ polynomials over $\reals$. \\

Let $\cA\eqdef(\cL(\lambda,X),N(X))$ be a consistent deterministic $p$-SCLI optimization algorithm and let $\quadab{A,\bb}\in\posdefun{d}{\Sigma}$ be a quadratic function. By consistency (see \thmref{thm:conv_correct}) we have 
\begin{align*}
	\cL(1,A)&=-NA	 
\end{align*}
(for brevity we omit the functional dependency on $X$). Since coefficient matrices are assumed to be simultaneously triangularizable, there exists an invertible matrix $Q\in\reals^{d\times d}$ such that $$T_j\eqdef Q^{-1} C_j Q,\quad j=0,1,\dots,p-1$$ are upper triangular matrices. Thus, by the definition of the characteristic polynomial (Definition \ref{def:l_lambda}) we have,
\begin{align} \label{decomp_sp}
		\det \syspola  &= \det\circpar{Q^{-1} \syspola Q } = \det\circpar{I_d\lambda^p - \sum_{j=0}^{p-1} T_j\lambda^{j}} = \prod_{j=1}^d \ell_j(\lambda)
\end{align}
where
\begin{align} 
 \label{eq:system_polys}
	\ell_j(\lambda) &= \lambda^p - \sum_{k=0}^{p-1} \sigma_j^k\lambda^{k} 
\end{align}
and where $\sigma_1^j,\dots, \sigma_d^{j},~j=0,\dots,p-1$ denote the elements on the diagonal of $T_j$, or equivalently the eigenvalues of $C_j$ ordered according to $Q$. Hence, the root radius of the characteristic polynomial of $\cA$ is
\begin{align} \label{eq:max_over_ells}
	\rho_\lambda(\syspola) &=  \max {\myset{\absval{\lambda}}{\ell_i(\lambda)=0 \text{ for some } i\in[d]}} 
\end{align}
On the other hand, by consistency condition (\ref{consis_1_syspol}) we get that for all $i\in[d]$ 
\begin{align} \label{constraint_sp}
	\ell_i(1)=\sigma_i\circpar{\syspol{}{1}}=\sigma_i\circpar{-NA} 
\end{align}
It remains to derive a lower bound on the maximum modulus of the roots of $\ell_i(\lambda)$, subject to the constraint (\ref{constraint_sp}). To this end, we employ the following lemma whose proof can be found in Section \ref{proof:lem:comp_poly}. 
\begin{lemma} \label{lem:comp_poly}
Suppose $q(z)$ is a real monic polynomial of degree $p$. If $q(1)<0$, then $$\rho(q(z))>1$$ Otherwise, if $q(1)\ge0$, then 
$$\rho(q(z))\ge\absval{\!\sqrt[q]{q(1)}-1\!}$$
In which case, equality holds if and only if $$q(z) =	\circpar{z-(1-\sqrt[p]{q(1)})}^p$$
\end{lemma}
We remark that the second part of \lemref{lem:comp_poly} implies that subject to constraint (\ref{constraint_sp}), the lower bound stated above is unimprovable. This property is used in Section \ref{section:ub} where we aim to obtain optimal $p$-SCLIs by designing $\ell_j(\lambda)$ accordingly. Clearly, in the presence of additional constraints, one might be able to improve on this lower bound (see Section \ref{section:is_this_tight}).\\

Since $\cA$ is assumed to be consistent, \lemref{lem:comp_poly} implies that $\spec{-N(A)A}\subseteq\reals^+$, as well as the following lower bound on the root radius of the characteristic polynomial,
\begin{align} \label{nice_one}
	\rho_\lambda(\syspola)&\ge \max_{i\in[d]} \absval{\sqrt[p]{\sigma_i(-N(A)A)}-1} 
\end{align}
Noticing that the reasoning above can be readily applied to stochastic $p$-SCLI optimization algorithms, we arrive at the following corollary which combines \thmref{thm:ic_cli} and Inequality (\ref{nice_one}).
\begin{corollary} \label{cor:inv_std_pol}
Let $\cA$ be a consistent $p$-SCLI optimization algorithm with respect to some $A\in\posdef{d}{\Sigma}$, let $N(X)$ denote the corresponding inversion matrix and let  
\begin{align*}
	\rho^*&= \max_{i\in[d]} \absval{\sqrt[p]{\sigma_i(-\bE N(A)A)}-1} 
\end{align*}
then the iteration complexity of $\cA$ for any $\quadab{A,\bb}\in\posdefun{d}{\Sigma}$ is lower bounded by 
\begin{align} \label{cor:bound_by_inv}
	\tilde{\Omega}\circpar{\frac{\rho^*}{1-\rho^*}\ln(1/\epsilon)} 
\end{align}

\end{corollary}

Using \corref{cor:inv_std_pol}, we are now able to provide a concise 'plug-and-play' scheme for deriving lower bounds on the iteration complexity of $p$-SCLI optimization algorithms. To motivate this scheme, note that the effectiveness of the lower bound stated in \corref{cor:inv_std_pol} is directly related to the magnitude of the eigenvalues of $-N(X)X$. To exemplify this, consider the inversion matrix of Newton method (see Section \ref{section_spec_algo}) 
\begin{align*}
N(X)=-X^{-1}
\end{align*}
Since $$\spec{-N(X)X}=\{1\}$$  the lower bound stated above is meaningless for this case. Nevertheless, the best computational cost for computing the inverse of $d\times d$ regular matrices known today is super-quadratic in $d$. As a result, this method might become impractical in large scale scenarios where the dimension of the problem space is large enough. A possible solution is to employ inversion matrices whose dependence on $X$ is simpler. On the other hand, if $N(X)$ approximates $-X^{-1}$ very badly, then the root radius of the characteristic polynomial might get too large. For instance, if $N(X)=0$ then $$\spec{-N(X)X}=\{0\}$$ contradicting the consistency assumption, regardless of the choice of the coefficient matrices. In the light of this, many optimization algorithms can be seen as strategies for balancing the computational cost of obtaining a good approximation for the inverse of $X$ and executing large number of iterations. Put differently, various structural restrictions on the inversion matrix yield different $\spec{-N(X)X}$, which in turn lead to a lower bound on the root radius of the corresponding characteristic polynomial. This gives rise to the following scheme: \\

\begin{table}[H] 
    \centering
    \begin{tabular*}{0.95\textwidth}{ll} 
        \toprule
        \textbf{Scheme 1 }&  Lower bounds \\
        \midrule
				\textbf{Parameters:} &$\bullet$ A family of quadratic functions $\posdefun{d}{\Sigma}$\\&$\bullet$ An inversion matrix $N(X)$ \\& $\bullet$ A lifting factor $p\in\bN$,\\
								
				\textbf{Choose }& $\cS'\subseteq \posdef{d}{\Sigma}$\\
				\textbf{Verify}& $\forall A\in\cS',~\spec{-\bE N(A)A}\subseteq \circpar{0,-2^p}$ for consistency\\ 
				\textbf{Bound} &$\displaystyle\max_{A\in\cS',i\in[d]} \absval{\sqrt[p]{\sigma_i(-\bE N(A)A)}-1}$ from below by some $\rho_*\in[0,1)$\\
				\midrule
        \textbf{Lower bound: }&   $\tilde{\Omega}\circpar{\frac{\rho_*}{1-\rho_*}\ln(1/\epsilon)}$			\\
        \bottomrule
    \end{tabular*}
\end{table}
This scheme is implicitly used in the previous Section (\ref{section:odc}), where we established a lower bound on the convergence rate of $2$-SCLI optimization algorithms over $\reals$ with constant inversion matrix and the following parameters
\begin{align*}
	\Sigma = [\mu,L],\quad \cS'=\{ \mu,L\}
\end{align*} 
In Section \ref{chapter:lower_bounds} we will make this scheme concrete for scalar and diagonal inversion matrices. 

\subsection{Bounds Schemes} \label{subsection:gen_bounds}
In spite of the fact that Scheme 1 is expressive enough for producing meaningful lower bounds under various structures of the inversion matrix, it does not allow one to incorporate other lower bounds on the root radius of characteristic polynomials whose coefficient matrices admit a particular form, e.g., linear coefficient matrices (see \ref{definition:first_linear_coefficient_matrices} below). Abstracting away from Scheme 1, we now formalize one of the main pillar of this work, i.e., the relation between the amount of computational cost one is willing to invest in executing each iteration and the total number of iterations needed for obtaining a given level of accuracy. 
We use this relation to form two schemes for establishing lower and upper bounds for $p$-SCLIs.\\

Given a compatible set of parameters: a lifting factor $p$, an inversion matrix $N(X)$, set of quadratic functions $\posdefun{d}{\Sigma}$ and a set of coefficients matrices $\cC$, we denote by $\mathfrak{A}(p,N(X),\posdefun{d}{\Sigma}, \cC)$ the set of consistent $p$-SCLI optimization algorithms for $\posdefun{d}{\Sigma}$ whose inversion matrix are $N(X)$ and whose coefficient matrices are taken from $\cC$. Furthermore, we denote by $\mathfrak{L}(p,N(X),\posdefun{d}{\Sigma}, \cC)$ the following set of polynomial matrices  
\begin{align*}
\left\{\syspola\eqdef I_d\lambda^p - \sum_{j=0}^{p-1} \bE C_j(X) 	\lambda^{j}\middle|~C_j(X)\in\cC,~ \syspol{}{1,A}=-N(A)A,~\forall A\in\posdef{d}{\Sigma} \right\}
\end{align*}
Since both sets are determined by the same set of parameters, the specifications of which will be occasionally omitted for brevity. The natural one-to-one correspondence between these two set, as manifested by \thmref{thm:ic_cli} and \corref{thm:conv_correct}, yields
\begin{equation} \label{eq:pillar}
	\boxed{\min_{\cA\in\mathfrak{A}} ~\max_{\quadab{A,\bb}\in\posdefun{d}{\Sigma}} \rho_\lambda(\syspol{\cA}{\lambda,A}) = \min_{\cL(\lambda,X)\in\mathfrak{L}} ~\max_{A\in\posdef{d}{\Sigma}} \rho_\lambda(\cL(\lambda,A))}
\end{equation}
The importance of \eqref{eq:pillar} stems from its ability to incorporate any bound on the maximal modulus root of polynomial matrices into a general scheme for bounding the iteration complexity of $p$-SCLIs. This is summarized by the following scheme.
\begin{table}[H] 
    \centering
    \begin{tabular*}{\textwidth}{lll} 
        \toprule
        \textbf{Scheme 2}&  Lower bounds \\
        \midrule
				\textbf{Given} &	a set of $p$-SCLI optimization algorithms $\mathfrak{A}(p,N(X),\posdefun{d}{\Sigma}, \cC)$\\
				\textbf{Find} &$\rho_*\in[0,1)$ such that \\	
				&$\qquad\displaystyle \min_{\syspol{}{\lambda,X} \in\mathfrak{L}} ~\max_{A\in \posdef{d}{\Sigma} } \rho_\lambda \circpar{{\syspol{}{\lambda,A}}}\ge\rho_*$\\

											\midrule
        \textbf{Lower bound: }&   $\tilde{\Omega}\circpar{\frac{\rho_*}{1-\rho_*}\ln(1/\epsilon)}$			\\
        \bottomrule
    \end{tabular*}
\end{table}
Thus, Scheme 1 is in effect an instantiation of the scheme shown above using \lemref{lem:comp_poly}. This correspondence of $p$-SCLI optimization algorithms and polynomial matrices can be also used contrariwise to derive efficient algorithm optimization. Indeed, in Section \ref{section_spec_algo} we show how FGD, HB and AGD can be formed as optimal instantiations of the following dual scheme.
\begin{table}[H] 
    \centering
    \begin{tabular*}{\textwidth}{lll} 
        \toprule
        \textbf{Scheme 3}&  Optimal $p$-SCLI Optimization Algorithms \\
        \midrule
				\textbf{Given} &	a set of polynomial matrices $\mathfrak{L}(p,N(X),\posdefun{d}{\Sigma}, \cC)$\\
				\textbf{Compute} & $\rho^*=\displaystyle \min_{\syspol{}{\lambda,X} \in\mathfrak{L}} ~\max_{A\in \posdef{d}{\Sigma} } \rho_\lambda \circpar{\syspol{}{\lambda,A}}$ \\&and denote its minimizer by $\cL^*\circpar{\lambda,A}$\\

											\midrule
        \textbf{Upper bound: }&  The corresponding $p$-SCLI algorithm of $\cL^*\circpar{\lambda,A}$\\
				\textbf{Convergence rate:} &$\bigO{\frac{1}{1-\rho^*}\ln(1/\epsilon)}$			\\
        \bottomrule
    \end{tabular*}
\end{table}

		\section{Lower Bounds} \label{chapter:lower_bounds}

In the sequel we derive lower bounds on the convergence rate of $p$-SCLI optimization algorithms whose inversion matrices are scalar or diagonal, and discuss the assumptions under which these lower bounds meet matching upper bounds. It is likely that this approach can be also effectively applied for block-diagonal inversion, as well as for a much wider set of inversion matrices whose entries depend on a relatively small set of entries of the matrix to be inverted.\\

\subsection{Scalar and Diagonal Inversion Matrices} \label{section_non_dep}
We derive a lower bound on the convergence rate of $p$-SCLI optimization algorithms for $L$-smooth $\mu$-strongly convex functions over $\reals^d$ with a scalar inversion matrix $N(X)$ by employing Scheme 1 (see Section \ref{subsection:general_case}). Note that since the one-dimensional case was already proven in Section \ref{section:odc}, we may assume that $d\ge2$.\\

First, we need to pick a `hard' matrix in $\posdef{d}{[\mu,L]}$. It turns out that any positive-definite matrix $A\in\posdef{d}{[\mu,L]}$ for which
\begin{align}\label{assump:full_spec_e}
	\set{\mu,L}\subseteq \spec{A} 
\end{align}
will meet this criterion. For the sake of concreteness, let us define 
\begin{align*}
	A\eqdef\operatorname{Diag}(L,\underbrace{\mu,\dots,\mu}_{d-1\text{ times}})
\end{align*}
In which case, 
\begin{align*} 
-\nu\{\mu,L\} = \spec{-\bE N(A)A}
\end{align*}
where $\nu\eqdef\bE [N(A)]$. Thus, to maintain consistency, it must hold that\footnote{On a side note, this reasoning also implies that if the spectrum of a given matrix $A$ contains both positive and negative eigenvalues then $A^{-1}b$ cannot be computed using $p$-SCLIs with scalar inversion matrices.} 
\begin{align} \label{nu_good_range}
\nu\in\circpar{\frac{-2^p}{L},0}
\end{align}
Next, to bound from below 
\begin{align*}
	\rho_* \eqdef \max_{i\in[d]} \absval{\sqrt[p]{\sigma_i(- \nu A)}-1} = \max\left\{|\sqrt[p]{- \nu \mu}-1|,
	|\sqrt[p]{- \nu L}-1| \right\}	
\end{align*}
we split the feasible range of $\nu$ (\ref{nu_good_range}) into three different sub-ranges as follows: 
\begin{table}[H] 
		\centering
    \begin{tabular}{l|l|l} 
		& $\sqrt[p]{- \nu \mu}-1 < 0$ & $\sqrt[p]{- \nu \mu}-1 \ge 0$ \\
		\hline
		& \underline{Case 1}& N/A \\
		$\sqrt[p]{- \nu L}-1 \le 0$  & Range: $[-1/L,0)$  & \\
		& Minimizer: $\nu^*=-1/L$ &\\
		& Lower bounds: $1-\sqrt[p]{\frac{\mu}{L}}$ &\\
		\hline
		&\underline{Case 2} & \underline{Case 3 (requires: $p\ge \log_2\kappa$)}\\
		$\sqrt[p]{- \nu L}-1 > 0$& Range: $(-1/\mu,-1/L)$  & Range: $(-2^p/L,-1/\mu]$\\
		 & Minimizer: $-\circpar{\frac2{\sqrt[p]{ L} +\sqrt[p]{\mu} }}^p $ &Minimizer: $-1/\mu$ \\
		& Lower bound: $\frac{\sqrt[p]{ L/\mu} -1 }{\sqrt[p]{L/\mu} +1}$ &Lower Bound: $\sqrt[p]{\frac{L}{\mu}}-1$
		\end{tabular}
		\caption{Lower bound for $\rho_*$ by subranges of $\nu$} \label{table:nu_subranges}
\end{table}
Therefore,
\begin{align} \label{ineq:spec_all_in_all}
\rho_* \ge \min \left\{ 
1-\sqrt[p]{\frac{\mu}{L}},
\frac{\sqrt[p]{ L/\mu} -1 }{\sqrt[p]{L/\mu} +1},
\sqrt[p]{\frac{L}{\mu}}-1
\right\} = \frac{\sqrt[p]{\kappa}-1}{\sqrt[p]{\kappa}+1}
\end{align}
Where $\kappa\eqdef L/\mu$, upper bounds the condition number of functions in $\posdefun{d}{[\mu,L]}$. Thus, by Scheme 1 we get the following lower bound on the worse-case iteration complexity, 
\begin{align} 
\tilde{\Omega}\circpar{\frac{\sqrt[p]{\kappa}-1}{2}\ln(1/\epsilon)}
\end{align}
As for the diagonal case, it turns out that for any quadratic $\quadab{A,b}\in\posdefun{d}{[\mu,L]}$ which has 
\begin{align}
\mymat{ \frac{L+\mu}{2} & \frac{L-\mu}{2} \\ \frac{L-\mu}{2} & \frac{L+\mu}{2}}
\end{align}
 as a principal sub-matrix of $A$, the best $p$-SCLI optimization algorithm with a diagonal inversion matrix does not improve on the optimal asymptotic convergence rate achieved by scalar inversion matrices (see Section \ref{ap:reduction}). Overall, we obtain the following theorem.
\begin{theorem} \label{thm:lb_ic_dia}
Let $\cA$ be a consistent $p$-SCLI optimization algorithm for $L$-smooth $\mu$-strongly convex functions over $\reals^d$. If the inversion matrix of $\cA$ is diagonal, then there exists a quadratic function $\quadab{A,\bb}\in\posdefun{d}{[\mu,L]}$ such that 
\begin{align} 
\IC_\cA\circpar{\epsilon,\quadab{A,\bb}}=\tilde{\Omega}\circpar{\frac{\sqrt[p]{\kappa}-1}{2}\ln(1/\epsilon)}
\end{align}
where $\kappa = L/\mu$.
\end{theorem}

\subsection{Is This Lower Bound Tight?} \label{section:is_this_tight}
A natural question now arises: is the lower bound stated in \thmref{thm:lb_ic_dia} tight? In short, it turns out that for $p=1$ and $p=2$ the answer is positive. For $p>2$ the answer heavily depends on whether a suitable spectral decomposition is within reach. Obviously, computing the spectral decomposition for a given positive definite matrix $A$ is at least as hard as finding the minimizer of a quadratic function whose hessian is $A$. To avoid this, we will later restrict our attention to linear coefficients matrices which allow efficient implementation. 	
\begin{description}
\item [A matching upper bound for $p=1$] In this case the lower bound stated in \thmref{thm:lb_ic_dia} is simply attained by FGD (see Section \ref{spec:fgd}). 
\item [A matching upper bound for $p=2$] In this case there are two 2-SCLI optimization algorithm which attain this bound, namely, Accelerated Gradient Descent and The Heavy Ball method (see Section \ref{spec:agd}), whose inversion matrices are scalar and correspond to Case 1 and Case 2 in Table \ref{table:nu_subranges}, i.e.,
\begin{align*}
N_{\text{HB}} &= -\circpar{\frac{2}{\sqrt{L}+\sqrt{\mu}}}^2I_d,\quad
N_{\text{AGD}} = \frac{-1}{L}I_d 
\end{align*}
Although HB obtains the best possible convergence rate in the class of 2-SCLIs with diagonal inversion matrices, it has a major disadvantage. When applied on general smooth and strongly-convex functions, one cannot guarantee global convergence. That is, in order to converge correctly, HB must be initialized close enough to the minimizer (see Section 3.2.1 in \cite{polyak1987introduction}). Indeed, if the initialization point is too far from the minimizer then HB may diverge as shown in Section 4.5 in \cite{lessard2014analysis}.  In contrast to this, AGD attains a global linear convergence with a slightly worse factor. Put differently, the fact HB is highly adapted to quadratic functions prevents it from converging globally to the minimizers of general smooth and strongly convex functions.

\item [A matching upper bound for $p>2$] In Subsection \ref{section:new_algo} we show that when no restriction on the coefficient matrices is imposed, the lower bound shown in \thmref{thm:lb_ic_dia} is tight, i.e., for any $p\in\bN$ there exists a matching $p$-SCLI optimization algorithm with scalar inversion matrix whose iteration complexity is
\begin{align}
\bigtO{\sqrt[p]{\kappa}\ln(1/\epsilon)}
\end{align}
In light of the existing lower bound which scales according to  $\sqrt{\kappa}$, this result may seem surprising at first. However, there is a major flaw in implementing these seemingly ideal $p$-SCLIs. In order to compute the corresponding coefficients matrices one has to obtain a very good approximation for the spectral decomposition of the positive definite matrix which defines the optimization problem. Clearly, this approach is rarely practical. To remedy this situation we focus on linear coefficient matrices
which admit a relatively low computational cost per iteration. That is, we assume that there exist real scalars $\alpha_1,\dots,\alpha_{p-1}$ and $\beta_1,\dots,\beta_{p-1}$ such that
\begin{align} \label{definition:first_linear_coefficient_matrices}
C_j(X) &= \alpha_j X + \beta_j I_d,\quad j=0,1,\dots,p-1
\end{align}
We believe that for these type of coefficient matrices the lower bound derived in \thmref{thm:lb_ic_dia} is not tight. Precisely, we conjecture that for any $0<\mu<L$ and for any consistent $p$-SCLI optimization algorithm $\cA$ with diagonal inversion matrix and linear coefficients matrices, there exists $\quadab{A,\bb}\in\posdefun{d}{[\mu,L]}$ such that 
\begin{align*}
\rho_\lambda(\syspol{\cA}{\lambda,X}) \ge \frac{\sqrt{\kappa}-1}{\sqrt{\kappa}+1}
\end{align*}
where $\kappa\eqdef L/\mu$. Proving this may allow to derive tight lower bounds for many optimization algorithm in the field of Machine Learning, such as SAG (see Section \ref{section_spec_algo}), whose structure is often very close to that of $p$-SCLIs with linear coefficients matrices. By Scheme 2, this conjecture may be equivalently stated as follows: suppose $q(z)$ is a $p$-degree monic real polynomial such that $q(1)=0$. Then, for any polynomial $r(z)$ of degree $p-1$ and for any $0<\mu<L$, there exists $\eta\in[\mu,L]$ such that 
\begin{align*}
\rho(q(z) - \eta r(z) ) &\ge \frac{\sqrt{L/\mu} -1}{\sqrt{L/\mu}+1}
\end{align*}
That being so, can we do better if we allow families of quadratic functions $\posdefun{d}{\Sigma}$ where $\Sigma$ are not necessarily continuous intervals? It turns out that the answer is positive. Indeed, in Section \ref{section:Lift_factor_3} we present a $3$-SCLI optimization algorithm with linear coefficient matrices which, by being intimately adjusted to quadratic functions whose hessian admits large enough spectral gap, beats the lower bound of Nemirovsky and Yudin (\ref{ineq:sqrtlb}). This apparently contradicting result is also discussed in Section \ref{section:Lift_factor_3}, where we show that lower bound (\ref{ineq:sqrtlb}) is established by employing quadratic functions whose hessian admits spectrum which densely populates $[\mu,L]$.  We would like to stress that as useful as such optimization algorithm might be, it is provided only for the purpose of demonstrating the detailed analysis which this framework allows and for indicating that applications which exhibit spectrum that distribute differently in $[\mu,L]$ might admit faster general solvers than what is dictated by lower bound (\ref{ineq:sqrtlb}). 

\end{description}

			\section{Upper Bounds} \label{section:ub} 
Up to this point we have projected various optimization algorithms on this framework of $p$-SCLI optimization algorithms, thereby converting questions on convergence properties into questions on moduli of roots of polynomials. In what follows, we shall head in the opposite direction. That is to say, first we define a polynomial (see Definition (\ref{def:l_lambda})) which meets a prescribed set of constraints, and then we form the corresponding  $p$-SCLI optimization algorithm. As stressed in Section \ref{section:is_this_tight}, we will focus exclusively on linear coefficient matrices which admit a low per-iteration computational cost and allow a straightforward extension to general smooth and strongly convex functions. Surprisingly enough, this allows a systematic recovering of FGD, HB, AGD, as well as establishing new optimization algorithms which allow better utilization of second-order information. This line of inquiry is particularly important due to the obscure nature of AGD, and further emphasizes its algebraic characteristic. We defer stochastic coefficient matrices, as in SDCA, (Section \ref{section:sdca_case_study}) to future work.\\

This section is organized as follows: First we apply Scheme 3 to derive general $p$-SCLIs with linear coefficients matrices; Next, following this line of argument, we recover AGD and HB as optimal instantiations under this setting; Finally, although general $p$-SCLI algorithms are exclusively specified for quadratic functions, we show how $p$-SCLIs with linear coefficient matrices can be extended to general smooth and strongly convex functions.

\subsection{Linear Coefficient Matrices} \label{section:lin_coeff_mat}
In the sequel we instantiate Scheme 3 (see Section \ref{subsection:gen_bounds}) for  $\cC_{\text{Linear}}$, the family of deterministic linear coefficient matrices. \\

First, note that due to consistency constraints, inversion matrices of constant $p$-SCLIs with linear coefficient matrices must be either constant scalar matrices or else be computationally equivalent to $A^{-1}$. Therefore, since our motivation for resorting to linear coefficient matrices was efficiency, we can safely assume that $N(X)=\nu I_d$ for some $\nu\in(-2^p/L,0)$. Following Scheme 3, we now seek the optimal characteristic polynomial in $\mathfrak{L}(p,\nu I_d,\posdefun{d}{[\mu,L]}, \cC_{\text{Linear}})$ with a compatible set of parameters (see Section \ref{subsection:gen_bounds}). In the presence of linearity, the characteristic polynomials takes the following simplified form
\begin{align*}
	\syspol{}{\lambda,X} &= \lambda^p - \sum_{j=0}^{p-1} (a_j X + b_j I_d) \lambda^j,\quad a_j,b_j\in\reals
\end{align*}
By (\ref{eq:max_over_ells}) we have 
\begin{align*}
	\rho_\lambda(\syspol{}{\lambda,X})&= \max \myset{|\lambda|}{\exists i\in[d],~ \ell_i(\lambda)=0  }
\end{align*}
where $\ell_i(\lambda)$ denote the factors of the characteristic polynomial as in (\ref{eq:system_polys}). That is, denoting the eigenvalues of $X$ by $\sigma_1,\dots,\sigma_d$ we have
$$\ell_i(\lambda)=\lambda^p - \sum_{j=0}^{p-1} (a_j \sigma_i + b_j) \lambda^j=\lambda^p - \sigma_i \sum_{j=0}^{p-1} a_j \lambda^j+\sum_{j=0}^{p-1} b_j \lambda^j$$
Thus, we can express the maximal root radius of the characteristic polynomial over $\posdefun{d}{[\mu,L]}$ in terms of the following polynomial 
\begin{align} \label{definition:general_characteristic}
	\ell(\lambda,\eta) &= \lambda^p-(\eta a(\lambda)+b(\lambda)) 
\end{align}
for corresponding real univariate $p-1$ degree polynomials $a(\lambda)$ and $b(\lambda)$, whereby
\begin{align*}
	\max_{A\in \posdef{d}{\Sigma} } \rho_\lambda \circpar{\syspol{}{\lambda,A}}&=\max_{\eta\in[\mu,L]} \rho\circpar{ \ell(\lambda,\eta) }
\end{align*}
That being the case, finding the optimal characteristic polynomial in $\mathfrak{L}$ translates to the following minimization problem,
\begin{center}
\fbox{ 
 \addtolength{\linewidth}{-2\fboxsep}%
 \addtolength{\linewidth}{-2\fboxrule}%
\begin{minipage}{0.5\linewidth}\vspace{-1em}
	\begin{align}
			\underset{\ell(\lambda,\eta)\in\mathfrak{L}}{\text{minimize}} & ~\max_{\eta\in[\mu,L]} \rho_\lambda(\ell(\lambda,\eta))\nonumber\\
			\text{s.t. } & \ell(1,\eta) = -\nu\eta,\quad \eta \in [\mu,L]  \label{eq:lin_cor1} \\
						&\rho_\lambda(\ell(\lambda,\eta)) < 1  \label{eq:lin_cor2}
	\end{align}
 \end{minipage}
}				
\end{center}

(Note that in this case we think of $\mathfrak{L}$ as a set of polynomials whose variable assumes scalars). Let us calculate the optimal characteristic polynomial for the setting where the lifting factor is $p=1$, the family of quadratic functions under considerations is $\posdefun{d}{[\mu,L]}$ and the inversion matrix is $N(X)=\nu I_d,~\nu\in(-2/L,0)$. In which case (\ref{definition:general_characteristic}) takes the following form
\begin{align*}
\ell(\lambda,\eta) = \lambda - \eta a _0 - b_0
\end{align*}
where $a_0,b_0$ are some real scalars. In order to satisfy (\ref{eq:lin_cor1}) for all $\eta\in[\mu,L]$, we have no other choice but to set 
\begin{align*}
a_0 &= \nu,\quad b_0 = 1
\end{align*} 
which implies
\begin{align*}
	\rho_\lambda(\ell(\lambda,\eta)) = 1+\nu\eta 
\end{align*}
Since $\nu\in(-2/L,0)$, condition \ref{eq:lin_cor2} follows, as well. The corresponding 1-SCLI optimization algorithm is 
\begin{align*}
\bx^{k+1}=(I+\nu A) \bx^{k} + \nu \bb
\end{align*}
and its first-order extension (see Section \ref{section:f_o_e} below) is precisely FGD (see Section \ref{section_spec_algo}). Finally, note that the corresponding root radius is bounded from above by
\begin{align*}
\frac{\kappa}{\kappa+1}
\end{align*}
for $\nu=-1/L$, the minimizer in Case 2 of Table \ref{table:nu_subranges}, and by
\begin{align*}
\frac{\kappa-1}{\kappa+1}
\end{align*}
for $\nu=\frac{-2}{\mu+L}$, the minimizer in Case 3 of Table \ref{table:nu_subranges}. This proves that FGD is optimal for the class of 1-SCLIs with linear coefficient matrices. \figref{figure:FGD} shows how the root radius of the characteristic polynomial of FGD is related to the eigenvalues of the hessian of the quadratic function under consideration.
\begin{figure}[H] \label{figure:FGD}
  \centering
     \includegraphics[scale=0.6,trim= 0 240 0 250,clip]{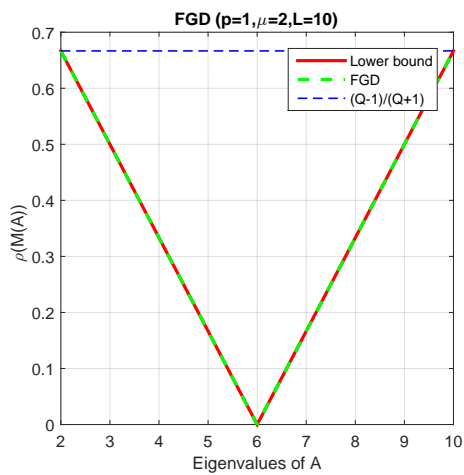}
  \caption{The root radius of FGD vs. various eigenvalues of the corresponding hessian.}
\end{figure}

\subsection{Recovering AGD and HB}
Let us now calculate the optimal characteristic polynomial for the setting where the lifting factor is $p=2$, the family of quadratic functions under considerations is $\posdefun{d}{[\mu,L]}$ and the inversion matrix is $N(X)=\nu I_d,~\nu\in(-4/L,0)$ (recall that the restricted range of $\nu$ is due to consistency). In which case (\ref{definition:general_characteristic}) takes the following form
\begin{align} \label{def:poly_q_p2}
\ell(\lambda,\eta)= \lambda^2 - \eta(a_1 \lambda + a_0 ) - (b_1 \lambda + b_0)
\end{align}
for some real scalars $a_0,a_1,b_0,b_1$. Our goal is to choose $a_0,a_1,b_0,b_1$ so as to minimize $$\max_{\eta\in[\mu,L]}\rho_\lambda(\ell(\lambda,\eta))$$ while maintaining conditions (\ref{eq:lin_cor1}) and (\ref{eq:lin_cor2}). Note that $\ell(\lambda,\eta)$, when seen as a function of $\eta$, forms a linear path of quadratic functions. Thus, a natural way to achieve this goal is to choose $\ell(\lambda,\eta)$ so that $\ell(\lambda,\mu)$ and $\ell(\lambda,L)$  take the form of the 'economic' polynomials introduced in \lemref{lem:comp_poly}, namely
\begin{align*}
\circpar{\lambda-(1-\sqrt{r})}^2
\end{align*}
for  $r=-\nu\mu$ and $r=-\nu L$, respectively, and hope that for others $\eta\in(\mu,L)$, the roots of $\ell(\lambda,\eta)$ would still be of  small magnitude. Note that due to the fact that $\ell(\lambda,\eta)$ is linear in $\eta$, condition (\ref{eq:lin_cor1}) readily holds for any $\eta\in(\mu,L)$. This yields the following two equations 
\begin{align*}
\ell(\lambda,\mu) &= \circpar{\lambda-(1-\sqrt{-\nu\mu)}}^2\\
\ell(\lambda,L) &= \circpar{\lambda-(1-\sqrt{-\nu L)}}^2
\end{align*}
Substituting (\ref{def:poly_q_p2}) for $\ell(\lambda,\eta)$ and expanding the r.h.s of the equations above we get
\begin{align*}
\lambda^2 - ( a_1\mu +b_1)\lambda - (a_0\mu + b_0) &= \lambda^2 -2(1-\sqrt{-\nu\mu})\lambda + (1-\sqrt{-\nu\mu})^2 \\
\lambda^2 - ( a_1  L +b_1)\lambda - (a_0 L+ b_0) &=  \lambda^2 -2(1-\sqrt{-\nu L})\lambda + (1-\sqrt{-\nu L})^2 
\end{align*}
Which can be equivalently expressed as the following system of linear equations
\begin{align}
- ( a_1\mu +b_1) &= -2(1-\sqrt{-\nu\mu}) \label{eq_params_1}\\
- (a_0 \mu+ b_0) &= (1-\sqrt{-\nu \mu})^2 \label{eq_params_2} \\
- ( a_1 L  +b_1) &= -2(1-\sqrt{-\nu L}) \label{eq_params_3}\\
- (a_0 L+ b_0) &= (1-\sqrt{-\nu L})^2  \label{eq_params_4}
\end{align}
Multiplying \eqref{eq_params_1} by -1 and add to it \eqref{eq_params_3}. Next, multiply \eqref{eq_params_2} by -1 and add to it \eqref{eq_params_4} yields
\begin{align*}
 a_1(\mu-L) &= 2\sqrt{-\nu}(\sqrt{L}-\sqrt{\mu})  \\
a_0 (\mu-L) &= (1-\sqrt{-\nu L})^2  -(1-\sqrt{-\nu \mu})^2 
\end{align*}
Thus,
\begin{align*}
a_1 &= \frac{-2\sqrt{-\nu} }{\sqrt{\mu}+\sqrt{L}},\qquad
a_0 = \frac{2\sqrt{-\nu}} {\sqrt{\mu}+\sqrt{L}} + \nu
\end{align*}
Remarkably enough, plugging in $\nu=-1/L$ (see Table \ref{table:nu_subranges}) into the equations above and solving for $b_1$ and $b_0$ yields a 2-SCLI optimization algorithm whose extension is precisely AGD (see Section \ref{section_spec_algo}). Following the same derivation only this time by setting $$\nu = -\circpar{\frac2{\sqrt{L}+\sqrt{\mu}}}^2$$ yields the Heavy-Ball method.\\  
\begin{figure}[H] \label{figure:AGD_HB}
  \centering
     $\left.\includegraphics[scale=0.6,trim= 100 240 150 250,clip]{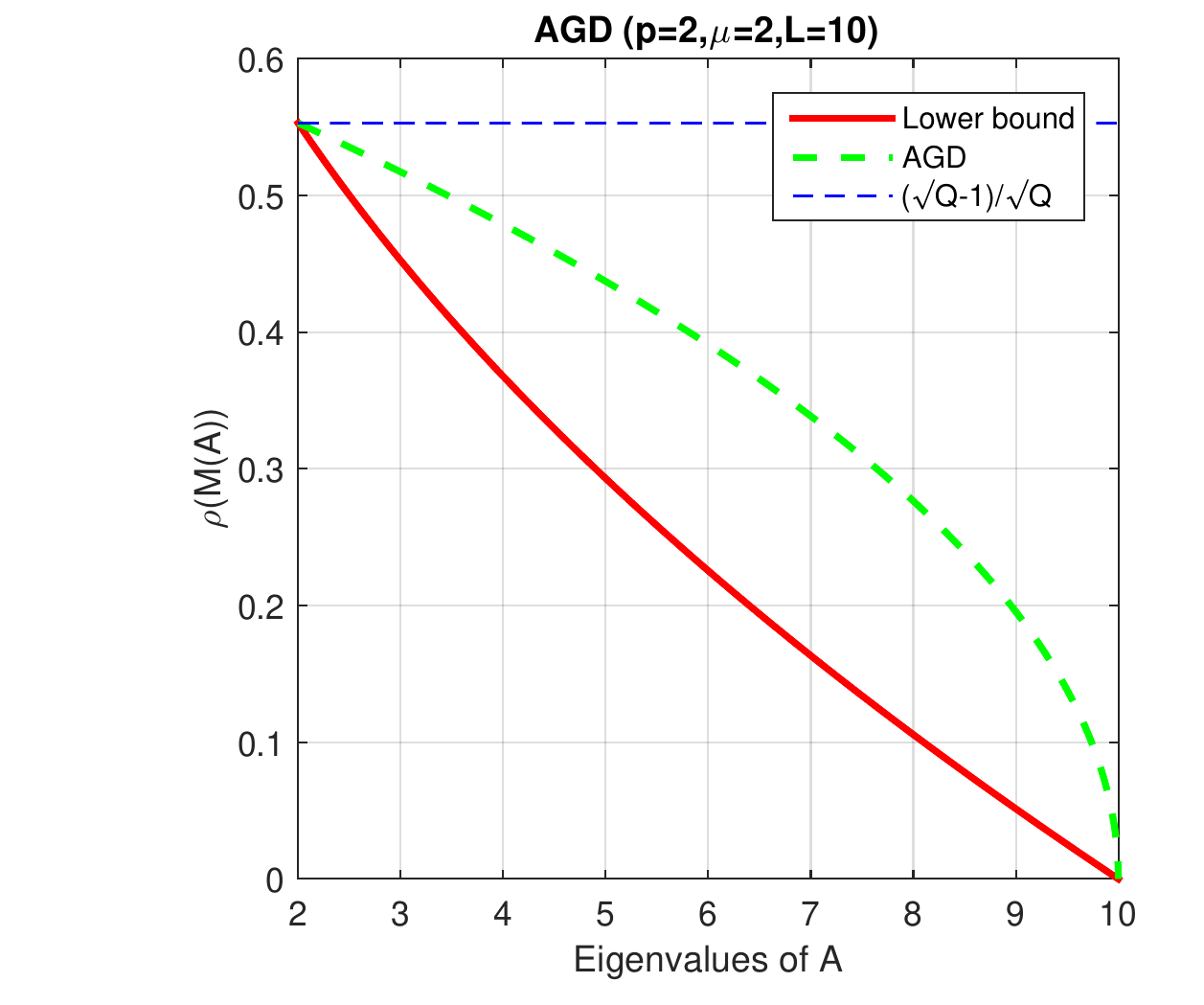}
		     \includegraphics[trim= 150 240 100 250,clip,scale=0.6]{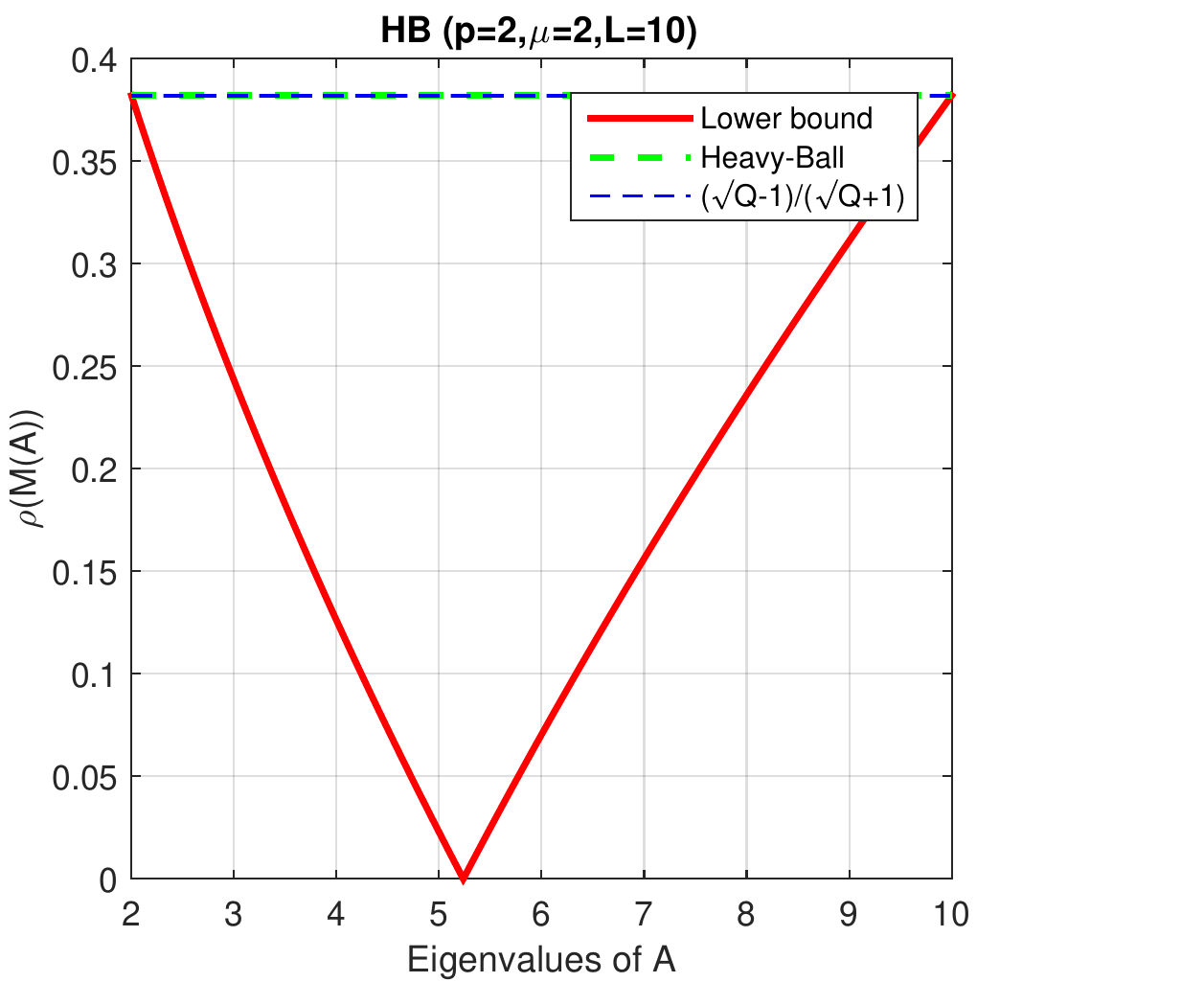}\right.$
  \caption{The root radius of AGD and HB vs. various eigenvalues of the corresponding hessian.}
\end{figure}
Moreover, using standard formulae for roots of quadratic polynomials one can easily verify that
\begin{align*}
\rho_\lambda\circpar{\ell(\lambda,\eta)}\leq \frac{\sqrt{\kappa}-1}{\sqrt{\kappa}},~\eta\in[\mu,L]
\end{align*}
for AGD, and 
\begin{align*}
\rho_\lambda\circpar{\ell(\lambda,\eta)}\leq \frac{\sqrt{\kappa}-1}{\sqrt{\kappa}+1},~\eta\in[\mu,L]
\end{align*}
for HB. In particular, Condition \ref{eq:lin_cor2} holds.  \figref{figure:AGD_HB} shows how the root radii of the characteristic polynomials of AGD and HB are related to the eigenvalues of the hessian of the quadratic function under consideration. 

\subsection{First-Order Extension for \texorpdfstring{$p$}{p}-SCLIs with Linear Coefficient Matrices} \label{section:f_o_e}
As mentioned before, as coefficient matrices of $p$-SCLIs can take any form, it is not clear how to use a given $p$-SCLI algorithm, efficient as it may be, for minimizing general smooth and strongly convex functions. That being the case, one could argue that recovering the specifications of, say, AGD for quadratic functions does not necessarily imply how to recover AGD itself. Fortunately, consistent $p$-SCLIs with linear coefficients can be reformulated as optimization algorithms for general smooth and strongly convex functions in a very natural way by substituting $\nabla f(\bx)$ for $A\bx+\bb$, while preserving the original convergence properties to a large extent. In the sequel we briefly discuss this appealing property, namely, canonical first-order extension, which completes the path from the world of polynomials to the world optimization algorithm for general smooth and strongly convex functions. \\

Let $\cA\eqdef(\syspol{\cA}{\lambda,X},N(X))$ be a consistent $p$-SCLI optimization algorithm with a scalar inversion matrix, i.e., $N(X) \eqdef \nu I_d,~\nu\in(-2^p/L,0)$, and linear coefficient matrices 
\begin{align} \label{def:lin_coeff}
C_j(X) = a_j X + b_j I_d ,~\quad j=0,\dots,p-1
\end{align}
where $a_0,\dots,a_{p-1}\in\reals$ and  $b_0,\dots,b_{p-1}\in\reals$ denote real scalars. Recall that by consistency, for any $\quadab{A,\bb}\in\posdefun{d}{\Sigma}$ it holds that
\begin{align*}
	\sum_{j=0}^{p-1}C_j(A) =  & I + \nu A
\end{align*} 
Thus
\begin{align} \label{eq:line_coeff_cons}
	\sum_{j=0}^{p-1} b_j &=1 \text{ and } \sum_{j=0}^{p-1} a_j =\nu
\end{align}
By the very definition of $p$-SCLI optimization algorithms (\defref{definition:pscli}), we have that 
\begin{align*}
	\bx^{k} = C_0(A) \bx^{k-p} + C_1(A) \bx^{k-(p-1)} + \dots +C_{p-1}(A) \bx^{k-1} + \nu \bb
\end{align*} 
Substituting $C_j(A)$ for (\ref{def:lin_coeff}), gives
\begin{align*}
	\bx^{k} = (a_0 A + b_0) \bx^{k-p} + (a_1 A + b_1)\bx^{k-(p-1)} + \dots +(a_{p-1} A + b_{p-1}) \bx^{k-1} + \nu \bb
\end{align*}
Rearranging and plugging in \ref{eq:line_coeff_cons}, we get 
\begin{align*}
	\bx^{k} &= 	a_0 (A \bx^{k-p} + \bb) + 	a_1 (A \bx^{k-(p-1)} + \bb) +\dots+ 
		a_{p-1} (A \bx^{k-1} + \bb)\\ &\quad+ 
	 b_0 \bx^{k-p}  + b_1 \bx^{k-(p-1)} + \dots + b_{p-1} \bx^{k-1} 
\end{align*}
Finally, by substituting $A\bx + \bb$ for its analog $\nabla f(\bx)$, we arrive at the following canonical first-order extension of $\cA$
\begin{align} \label{eq:nabla_count}
	\bx^{k} &=  \sum_{j=0}^{p-1} b_j \bx^{k-(p-j)} + \sum_{j=0}^{p-1} a_j \nabla f(\bx^{k-(p-j)})
\end{align}
Being applicable to a much wider collection of functions, how well should we expect the canonical extensions to behave? The answer is that when initialized close enough to the minimizer, one should expect a linear convergence of essentially the same rate. A formal statement is given by the theorem below which easily follows from Theorem 1 in Section 2.1, \cite{polyak1987introduction} for
\begin{align*}
  g(\bx^{k-p},\bx^{k-(p-1)},\dots,\bx^{k-1}) &=  \sum_{j=0}^{p-1} b_j \bx^{k-(p-j)} + \sum_{j=0}^{p-1} a_j \nabla f(\bx^{k-(p-j)})
\end{align*}
\begin{theorem}
Suppose $f:\reals^d\to\reals$ is an $L$-smooth $\mu$-strongly convex function and let $\bx^*$ denotes its minimizer. Then, for every $\epsilon>0$, there exist $\delta>0$ and $C>0$ such that if 
\begin{align*}
	\norm{\bx^j-\bx^*}\le \delta,\quad j=0,\dots,p-1
\end{align*}
then
\begin{align*}
	\norm{\bx^k-\bx^0}\le C(\rho^* + \epsilon)^k,\quad k =p,p+1,\dots
\end{align*}
where 
\begin{align*}
\rho^* = \sup_{\eta\in\Sigma} \rho\circpar{\lambda^p - \sum_{j=0}^{p-1} (a_j \eta + b_j) \lambda^j}
\end{align*}
\end{theorem}

Unlike general $p$-SCLIs with linear coefficient matrices which are guaranteed to converge only when initialized close enough to the minimizer, AGD converges linearly, regardless of the initialization points, for any smooth and strongly convex function. This merits further investigation as to the precise principles which underlie $p$-SCLIs of this kind.

\bibliography{thesis_bib}

\begin{thebibliography}{31}
\providecommand{\natexlab}[1]{#1}
\providecommand{\url}[1]{\texttt{#1}}
\expandafter\ifx\csname urlstyle\endcsname\relax
  \providecommand{\doi}[1]{doi: #1}\else
  \providecommand{\doi}{doi: \begingroup \urlstyle{rm}\Url}\fi

\bibitem[Allen-Zhu and Orecchia(2014)]{allen2014novel}
Zeyuan Allen-Zhu and Lorenzo Orecchia.
\newblock A novel, simple interpretation of nesterov's accelerated method as a
  combination of gradient and mirror descent.
\newblock \emph{arXiv preprint arXiv:1407.1537}, 2014.

\bibitem[Baes(2009)]{baes2009estimate}
Michel Baes.
\newblock Estimate sequence methods: extensions and approximations.
\newblock 2009.

\bibitem[Beck and Teboulle(2009)]{beck2009fast}
Amir Beck and Marc Teboulle.
\newblock A fast iterative shrinkage-thresholding algorithm for linear inverse
  problems.
\newblock \emph{SIAM Journal on Imaging Sciences}, 2\penalty0 (1):\penalty0
  183--202, 2009.

\bibitem[Drazin et~al.(1951)Drazin, Dungey, and Gruenberg]{drazin1951some}
MP~Drazin, JW~Dungey, and KW~Gruenberg.
\newblock Some theorems on commutative matrices.
\newblock \emph{Journal of the London Mathematical Society}, 1\penalty0
  (3):\penalty0 221--228, 1951.

\bibitem[Fell(1980)]{fell1980zeros}
Harriet Fell.
\newblock On the zeros of convex combinations of polynomials.
\newblock \emph{Pacific Journal of Mathematics}, 89\penalty0 (1):\penalty0
  43--50, 1980.

\bibitem[Gohberg et~al.(2009)Gohberg, Lancaster, and Rodman]{gohberg2009matrix}
Israel Gohberg, Pnesteeter Lancaster, and Leiba Rodman.
\newblock \emph{Matrix polynomials}, volume~58.
\newblock SIAM, 2009.

\bibitem[Higham and Tisseur(2003)]{higham2003bounds}
Nicholas~J Higham and Fran{\c{c}}oise Tisseur.
\newblock Bounds for eigenvalues of matrix polynomials.
\newblock \emph{Linear algebra and its applications}, 358\penalty0
  (1):\penalty0 5--22, 2003.

\bibitem[Horne(1997)]{horne1997lower}
Bill~G Horne.
\newblock Lower bounds for the spectral radius of a matrix.
\newblock \emph{Linear algebra and its applications}, 263:\penalty0 261--273,
  1997.

\bibitem[Huang and Wang(2007)]{huang2007improving}
Ting-Zhu Huang and Lin Wang.
\newblock Improving bounds for eigenvalues of complex matrices using traces.
\newblock \emph{Linear Algebra and its Applications}, 426\penalty0
  (2):\penalty0 841--854, 2007.

\bibitem[Johnson and Zhang(2013)]{johnson2013accelerating}
Rie Johnson and Tong Zhang.
\newblock Accelerating stochastic gradient descent using predictive variance
  reduction.
\newblock In \emph{Advances in Neural Information Processing Systems}, pages
  315--323, 2013.

\bibitem[Kushner and Yin(2003)]{kushner2003stochastic}
Harold~J Kushner and George Yin.
\newblock \emph{Stochastic approximation and recursive algorithms and
  applications}, volume~35.
\newblock Springer, 2003.

\bibitem[Lessard et~al.(2014)Lessard, Recht, and Packard]{lessard2014analysis}
Laurent Lessard, Benjamin Recht, and Andrew Packard.
\newblock Analysis and design of optimization algorithms via integral quadratic
  constraints.
\newblock \emph{arXiv preprint arXiv:1408.3595}, 2014.

\bibitem[Marden(1966)]{marden1966geometry}
Morris Marden.
\newblock \emph{Geometry of polynomials}.
\newblock Number~3 in @. American Mathematical Soc., 1966.

\bibitem[Mason and Handscomb(2002)]{mason2002chebyshev}
John~C Mason and David~C Handscomb.
\newblock \emph{Chebyshev polynomials}.
\newblock CRC Press, 2002.

\bibitem[Milovanovi{\'c} and Rassias(2000)]{milovanovic2000distribution}
Gradimir~V Milovanovi{\'c} and Themistocles~M Rassias.
\newblock Distribution of zeros and inequalities for zeros of algebraic
  polynomials.
\newblock In \emph{Functional equations and inequalities}, pages 171--204.
  Springer, 2000.

\bibitem[Milovanovic et~al.(1994)Milovanovic, Mitrinovic, and
  Rassias]{milovanovic1994topics}
Gradimir~V Milovanovic, DS~Mitrinovic, and Th~M Rassias.
\newblock Topics in polynomials.
\newblock \emph{Extremal Problems, Inequalities, Zeros, World Scientific,
  Singapore}, 1994.

\bibitem[Nemirovski(2005)]{nemirovski2005efficient}
Arkadi Nemirovski.
\newblock Efficient methods in convex programming.
\newblock 2005.

\bibitem[Nemirovsky and Yudin(1983)]{nemirovskyproblem}
AS~Nemirovsky and DB~Yudin.
\newblock Problem complexity and method efficiency in optimization. 1983.
\newblock \emph{Willey-Interscience, New York}, 1983.

\bibitem[Nesterov(1983)]{nesterov1983method}
Yurii Nesterov.
\newblock \emph{A method of solving a convex programming problem with
  convergence rate O (1/k2)}.
\newblock @, 1983.

\bibitem[Nesterov(2004)]{nesterov2004introductory}
Yurii Nesterov.
\newblock \emph{Introductory lectures on convex optimization}, volume~87.
\newblock Springer Science \& Business Media, 2004.

\bibitem[Polyak(1987)]{polyak1987introduction}
Boris~T Polyak.
\newblock \emph{Introduction to optimization}.
\newblock Optimization Software New York, 1987.

\bibitem[Rahman and Schmeisser(2002)]{rahman2002analytic}
Qazi~Ibadur Rahman and Gerhard Schmeisser.
\newblock \emph{Analytic theory of polynomials}.
\newblock Number~26 in @. Oxford University Press, 2002.

\bibitem[Roux et~al.(2012)Roux, Schmidt, and Bach]{roux2012stochastic}
Nicolas~Le Roux, Mark Schmidt, and Francis Bach.
\newblock A stochastic gradient method with an exponential convergence rate for
  finite training sets.
\newblock \emph{arXiv preprint arXiv:1202.6258}, 2012.

\bibitem[Shalev-Shwartz and Zhang(2013{\natexlab{a}})]{shalev2013accelerated}
Shai Shalev-Shwartz and Tong Zhang.
\newblock Accelerated proximal stochastic dual coordinate ascent for
  regularized loss minimization.
\newblock \emph{arXiv preprint arXiv:1309.2375}, 2013{\natexlab{a}}.

\bibitem[Shalev-Shwartz and Zhang(2013{\natexlab{b}})]{shalev2013stochastic}
Shai Shalev-Shwartz and Tong Zhang.
\newblock Stochastic dual coordinate ascent methods for regularized loss.
\newblock \emph{The Journal of Machine Learning Research}, 14\penalty0
  (1):\penalty0 567--599, 2013{\natexlab{b}}.

\bibitem[Spall(2005)]{spall2005introduction}
James~C Spall.
\newblock \emph{Introduction to stochastic search and optimization: estimation,
  simulation, and control}, volume~65.
\newblock John Wiley \& Sons, 2005.

\bibitem[Sutskever et~al.(2013)Sutskever, Martens, Dahl, and
  Hinton]{sutskever2013importance}
Ilya Sutskever, James Martens, George Dahl, and Geoffrey Hinton.
\newblock On the importance of initialization and momentum in deep learning.
\newblock In \emph{Proceedings of the 30th International Conference on Machine
  Learning (ICML-13)}, pages 1139--1147, 2013.

\bibitem[Tseng(2008)]{tseng2008accelerated}
Paul Tseng.
\newblock On accelerated proximal gradient methods for convex-concave
  optimization. submitted to siam j.
\newblock \emph{J. Optim}, 2008.

\bibitem[Walsh(1922)]{walsh1922location}
JL~Walsh.
\newblock On the location of the roots of certain types of polynomials.
\newblock \emph{Transactions of the American Mathematical Society}, 24\penalty0
  (3):\penalty0 163--180, 1922.

\bibitem[Wolkowicz and Styan(1980)]{wolkowicz1980bounds}
Henry Wolkowicz and George~PH Styan.
\newblock Bounds for eigenvalues using traces.
\newblock \emph{Linear Algebra and Its Applications}, 29:\penalty0 471--506,
  1980.

\bibitem[Zhong and Huang(2008)]{zhong2008bounds}
Qin Zhong and Ting-Zhu Huang.
\newblock Bounds for the extreme eigenvalues using the trace and determinant.
\newblock \emph{Journal of Information and Computing Science}, 3\penalty0
  (2):\penalty0 118--124, 2008.

\end{thebibliography}


\backmatter
\begin{appendices}
\chapter{An appendix}

\section{Optimal \texorpdfstring{$p$}{p}-SCLI for Unconstrained Coefficient Matrices} \label{section:new_algo}
In the sequel we employ Scheme 3 (see Section \ref{subsection:gen_bounds}) to show that, when no constraints are imposed on the functional dependency of the coefficient matrices, the lower bound shown in \thmref{thm:lb_ic_dia} is tight. To this end, recall that in \lemref{lem:comp_poly} we have shown that the lower bound on the maximal modulus of roots of a polynomials which evaluate at $z=1$ to some $r\ge0$ is uniquely attained by the following polynomial
\begin{align*}
	q_r^*(z) \eqdef \circpar{z-(1-\sqrt[p]{r})}^p
\end{align*}
Thus, by choosing coefficients matrices which admit the same form, we obtain the optimal convergence rate as stated in \thmref{thm:lb_ic_dia}.\\

Concretely, let $p\in\bN$ be some lifting factor, let $N(X) =\nu I_d,~\nu\in(-2^p/L,0)$ be a fixed scalar matrix and let $\quadab{A,\bb}\in\posdefun{d}{\Sigma}$ be some quadratic function. 
\lemref{lem:comp_poly} implies that for each $\eta\in\spec{-\nu A}$ we need the corresponding factor of the characteristic polynomial to be 
\begin{align}\label{eq:eco_poly_coeff}
\ell_j(\lambda) &= \circpar{\lambda-(1-\sqrt[p]{\eta})}^p\nonumber	\\
&=\sum_{k=0}^p \binom{p}{k}\circpar{ \sqrt[p]{-\nu\eta} -1}^{p-k}  \lambda^{k}
\end{align}
This is easily accomplished using the spectral decomposition of $A$ by 
$$\Lambda\eqdef U^\top A U$$ 
where $U$ is an orthogonal matrix and $\Lambda$ is a diagonal matrix. Note that since $A$ is a positive definite matrix such a decomposition must always exist. We define $p$ coefficient matrices $C_0,C_1,\dots,C_{p-1}$ in accordance with \eqref{eq:eco_poly_coeff} as follows
\begin{align*}
C_k = U \mymat{-\binom{p}{k}\circpar{ \sqrt[p]{-\nu \Lambda_{11}} -1}^{p-k} \\ & -\binom{p}{k}\circpar{ \sqrt[p]{-\nu \Lambda_{22}} -1}^{p-k}\\ && \ddots  \\ &&& -\binom{p}{k}\circpar{ \sqrt[p]{-\nu \Lambda_{dd}} -1}^{p-k} }U^\top
\end{align*}
By using \thmref{thm:conv_correct}, it is easily verified that these coefficient matrices form a consistent $p$-SCLI optimization algorithm whose characteristic polynomial's root radius is
\begin{align*}
\max_{j=1,\dots,d} \absval{ \sqrt[p]{-\nu \mu_j} - 1} 
\end{align*}
Choosing  
\begin{align*}
\nu=-\circpar{\frac2{\sqrt[p]{ L} +\sqrt[p]{\mu} }}^p 
\end{align*}
according to Table \ref{table:nu_subranges}, produces an optimal $p$-SCLI optimization algorithm for this set of parameters. It is noteworthy that other suitable decompositions can be used for deriving optimal $p$-SCLIs, as well.\\

As a side note, since the cost of computing each iteration in $\in\reals^{pd}$ grows linearly with the lifting factor $p$, the optimal choice of $p$ with respect to the condition number $\kappa$ yields a $p$-SCLI optimization algorithm whose iteration complexity is $\Theta(\ln(\kappa)\ln(1/\epsilon))$. Clearly, this result is of theoretical interest only as this would require a spectral decomposition of $A$, which, if no other structural assumptions are imposed, is an even harder task than computing the minimizer of $\quadab{A,\bb}$.

\section{Lifting Factor \texorpdfstring{$\ge$}{>=} 3} \label{section:Lift_factor_3}
In Section \ref{section:is_this_tight} we conjecture that for any $p$-SCLI optimization algorithm $\cA\eqdef\circpar{\syspola,N(X)}$, with diagonal inversion matrix and linear coefficient matrices there exists some $A\in\posdefun{d}{[\mu,L]}$ such that 
\begin{align} \label{ineq:sqrt2_lb}
	\rho_\lambda(\syspola)&\ge \frac{\sqrt{\kappa} - 1}{\sqrt{\kappa} + 1}  
\end{align}
However, it may be possible to overcome this barrier by focusing on a subclass of $\posdefun{d}{[\mu,L]}$. Indeed, recall that the polynomial analogy of this conjecture states that for any monic real $p$ degree polynomial $q(z)$ such that $q(1)=0$ and for any polynomial $r(z)$ of degree $p-1$, there exists $\eta\in[\mu,L]$ such that 
\begin{align*}
\rho(q(z) - \eta r(z) ) &\ge \frac{\sqrt{\kappa} -1}{\sqrt{\kappa}+1}
\end{align*}
where $\kappa\eqdef L/\mu$. This implies that we may be able to tune $q(z)$ and $r(z)$ so as to obtain a convergence rate, which breaks \ineqref{ineq:sqrt2_lb}, for quadratic function whose Hessian's spectrum does not spread uniformly across $[\mu,L]$. \\

Let us demonstrate this idea for $p=3,\mu = 2$ and $L=100$. Following the exact same derivation used in the last section, let us pick
\begin{align*}
	q(z,\eta)& \eqdef z^p - (\eta a(z) + b(z))
\end{align*}
numerically, so that 
\begin{align*}
q(z,\mu) &= \circpar{z-(1-\sqrt[3]{-\nu\mu)}}^3\\
q(z,L) &= \circpar{z-(1-\sqrt[3]{-\nu\mu)}}^3
\end{align*}
where 
\begin{align*}
	\nu&=-\circpar{\frac2{\sqrt[3]{ L} +\sqrt[3]{\mu} }}^3
\end{align*}
The resulting 3-CLI optimization algorithm $\cA_3$ is
\begin{align*}
	\bx^k &= C_2(X)\bx^{k-1} + C_1(X)\bx^{k-2} + C_0(X)\bx^{k-3} + N(X)b
\end{align*}
where
\begin{align*}
C_0(X) &\approx  0.1958 I_d - 0.0038 X\\
C_1(X) &\approx  -0.9850 I_d \\
C_2(X) &\approx  1.7892 I_d - 0.0351 X\\
N(X)   &\approx  -0.0389 I_d
\end{align*}
It is noteworthy that as opposed to the algorithm described in Section \ref{section:new_algo}, when employing linear coefficient matrices no knowledge regarding the eigenvectors of $A$ is required. As each eigenvalue of the second-order derivative corresponds to a bound on the convergence rate, one can verify by Figure \ref{fig:A_3} that
\begin{align*}
	\rho\circpar{\syspol{\cA_3}{\lambda,X}} \le \frac{\sqrt[3]{\kappa} - 1}{\sqrt[3]{\kappa}} 
\end{align*} 
for any $X\in\posdefun{d}{[2,100]}$ which satisfies 
\begin{align*}
	\spec{A}&\subseteq \hat{\Sigma}\eqdef [2,2+\epsilon]\cup[100-\epsilon,100],\quad\epsilon\approx 1.5
\end{align*}
Thus, $\cA_3$ outperforms AGD for this family of quadratic functions.\\

\begin{figure}
  \centering
	\includegraphics[scale=0.6,trim= 0 190 0 210,clip]{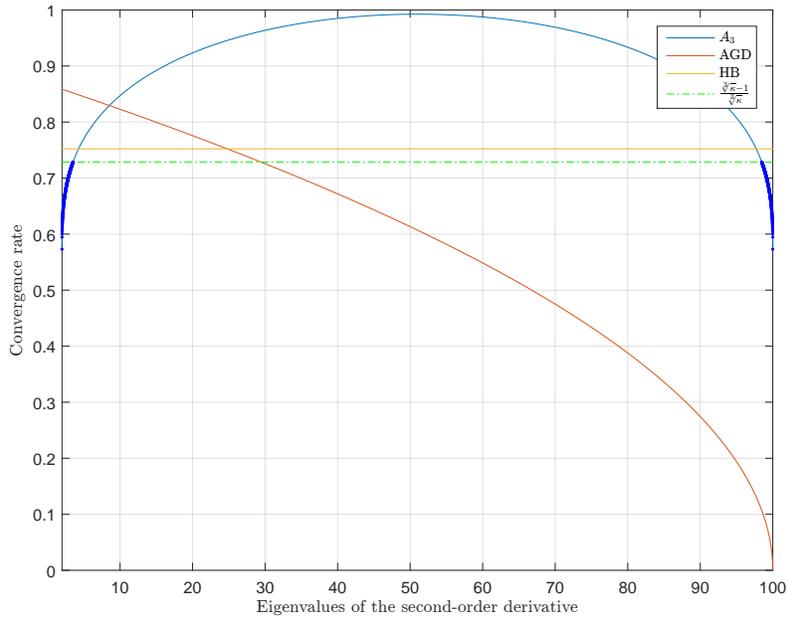}
  \caption{The convergence rate of AGD and $\cA_3$ vs. the eigenvalues of the second-order derivatives. It can be seen that the asymptotic convergence rate of $\cA_3$ for quadratic functions whose second-order derivative comprises eigenvalues which are close to the edges of $[2,100]$, is faster than AGD and goes below the theoretical lower bound for first-order optimization algorithm $\frac{\sqrt{\kappa}-1}{\sqrt{\kappa}+1}$.	  }
	\label{fig:A_3}
\end{figure}
Let us demonstrate the gain in the performance allowed by $\cA_3$ in a very simple setting. Define $A$ to be $\Diag{\mu,L}$ rotated counter-clockwise by $45^\circ$, that is
\begin{align*}
	A &= \mu \mymat{\frac{1}{\sqrt{2}}\\\frac{1}{\sqrt{2}} }
	\mymat{\frac{1}{\sqrt{2}}\\\frac{1}{\sqrt{2}} }^\top
	+
	L \mymat{\frac{1}{\sqrt{2}}\\\frac{-1}{\sqrt{2}} }
	\mymat{\frac{1}{\sqrt{2}}\\\frac{-1}{\sqrt{2}} }^\top
	= \mymat{ \frac{\mu+L}{2} & \frac{\mu-L}{2} \\ \frac{\mu-L}{2} & \frac{\mu+L}{2} }
\end{align*}
Furthermore, define  $\bb = -A \circpar{100,100}^\top$. Note that $\quadab{A,\bb}\in\posdefun{2}{\hat{\Sigma}}$ and that its minimizer is simply $\circpar{100,100}^\top$. Figure \ref{fig:err_A_3} shows the error of $\cA_3$, AGD and HB vs. iteration number. All algorithms are initialized at $\bx^0 =0$.
\begin{figure}[h]
\begin{center}
\includegraphics[scale=0.5,trim= 50 200 50 200,clip]{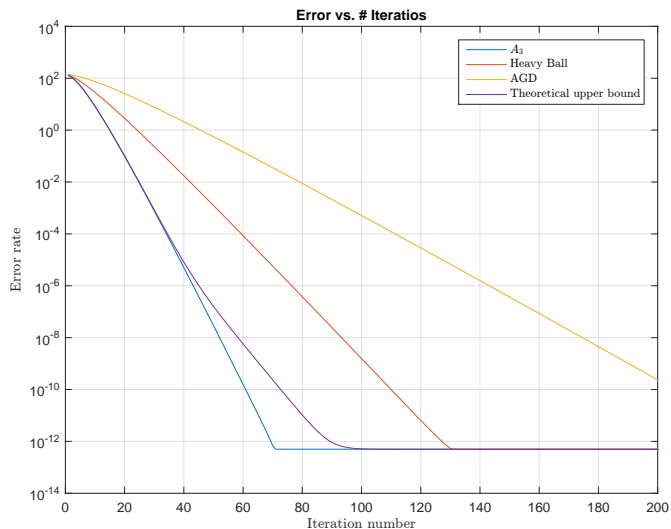}
\end{center}
  \caption{The error rate of $\cA_3$, AGD and HB vs. $\#$ iterations for solving a simple quadratic minimization task. The convergence rate of $\cA_3$ is bounded from above by $\frac{\sqrt[3]{\kappa}-1}{\sqrt[3]{\kappa}}$ as implied by theory.} 
	\label{fig:err_A_3}
\end{figure}
Since $\cA_3$ is a first-order optimization algorithm, by the lower bound shown in (\ref{ineq:sqrtlb}) there must exist some quadratic function $\quadab{A_{\text{lb}},\bb_{\text{lb}}}\in\posdefun{2}{[\mu,L]}$  such that
\begin{align} \label{iteration_complexity:general__lower_bound}
	\IC_{\cA_3}\circpar{\epsilon,\quadab{A_{\text{lb}},\bb_{\text{lb}}}}\ge \tilde{\Omega}\circpar{\sqrt{\kappa}\ln(1/\epsilon)}
\end{align}
But, since 
\begin{align}
	 \IC_{\cA_3}\circpar{\epsilon,\quadab{A,\bb}} \le \bigO{\sqrt[3]{\kappa}\ln(1/\epsilon)}  
\end{align}
for every $\quadab{A,\bb}\in\posdefun{2}{\hat{\Sigma}}$, we must have $\quadab{A_{\text{lb}},\bb_{\text{lb}}} \in \posdefun{2}{[\mu,L]} \setminus \posdefun{2}{\hat{\Sigma}}$. Indeed, in the somewhat simpler form of the general lower bound for first-order optimization algorithms, Nesterov (see \cite{nesterov2004introductory}) considers the following $1$-smooth $0$-strongly convex function\footnote{Although $\quadab{A_{\text{lb}},\bb_{\text{lb}}}$ is not strongly convex, the lower bound for strongly convex function is obtained by shifting the spectrum using a regularization term $\mu/2\norm{\bx}^2$. In which case, the shape of the spectrum is preserved.} 
\begin{align*}
	A_{\text{lb}} = \frac{1}{4} \mymat{
	2 & -1 & 0 & \dots && &0 \\
	 -1 & 2 & -1 & 0 & \dots && 0  \\
	0 & -1 & 2 & -1 & 0 & \dots & 0  \\\\
	&&&\ddots\\\\
	0&&\dots&0 &-1& 2 & -1  \\	
	0&&&\dots&0 &-1& 2   \\	
	},~\bb_{\text{lb}} = -\mymat{1\\0\\\vdots\\0}
\end{align*}
As demonstrated by Figure \ref{fig:NestSpec}, $\spec{A_{\text{lb}}}$ densely fills $[\mu,L]$.
\begin{figure}[H] 
\begin{center}
\includegraphics[scale=0.6,trim= 0 240 0 250,clip]{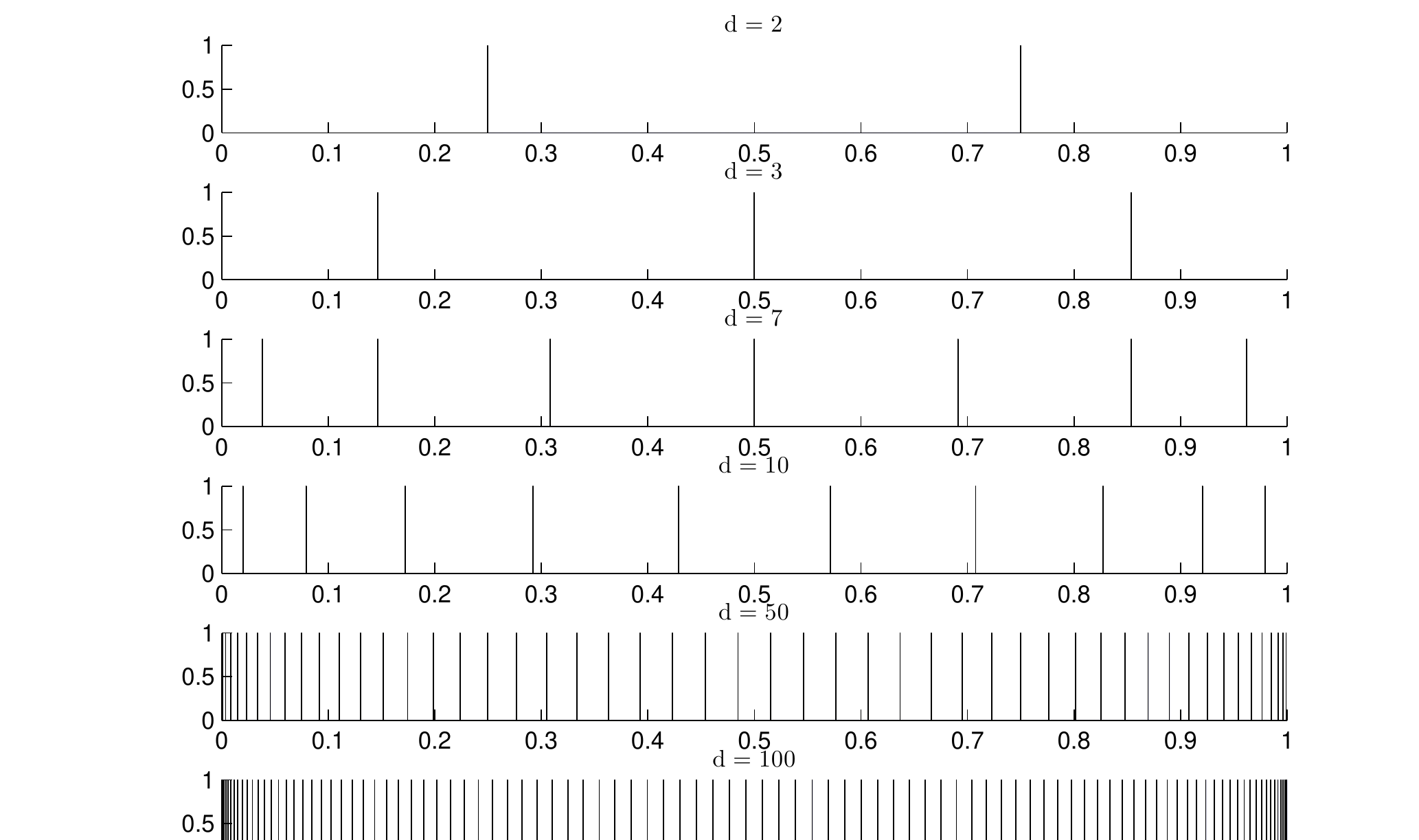}
\end{center}
  \caption{The spectrum of $A_{\text{lb}}$, as used in the derivation of Nesterov's lower bound, for problem space of various dimensions.}
	\label{fig:NestSpec}
\end{figure}

Consequently, we expect that whenever adjacent eigenvalues of the second-order derivatives are relatively distant, one should able be to minimize the corresponding quadratic function faster than the lower bound stated in \ref{ineq:sqrtlb}. This technique can be further generalized to $p>3$ using the same ideas. Also, a different approach is to use quadratic (or even higher degree) coefficient matrices to exploit other shapes of spectra. Clearly, the applicability of both approaches heavily depends the existence of spectra of this type in real applications.

\section{Proofs}

\subsection{Proof of \thmref{thm:ic_cli}} \label{subsection:conv_prop}
The simple idea behind proof of \thmref{thm:ic_cli} is to express the dynamic of a given $p$-SCLI optimization algorithm as a recurrent application of linear operator. To analyze the latter, we employ the Jordan form which allows us to bind together the maximal magnitude eigenvalue and the convergence rate. Prior to proving this theorem, we first need to introduce some elementary results in linear algebra.\\

\subsubsection{Linear Algebra Preliminaries}
We prove two basic lemmas which allow to determine under what conditions does a recurrence application of linear operators over finite dimensional spaces converge, as well as to compute the limit of matrices powers series. It is worth noting that despite of being a very elementary result in Matrix theory and in the theory of power methods, the lower bound part of the first lemma does not seem to appear in this form in standard linear algebra literature. 
\begin{lemma} \label{lem:conv_rate_jord}
Let $A$ be a $d\times d$ square matrix.
\begin{itemize}
	\item If $\rho(A)>0$ then there exists $C_A>0$  such that for any $\bu\in\reals^d$ and for any $k\in\bN$ we have  
\begin{align*}
\norm{A^k \bu} \le C_A k^{m-1} \rho(A)^k\norm{\bu} 
\end{align*}
where $m$ denotes the maximal index of eigenvalues whose modulus is maximal. \\In addition, there exists $c_A >0$ and $\br\in\reals^d$ such that for any $\bu\in\reals^d$ which satisfies $\inprod{\bu}{\br}\neq0$ we have 
\begin{align*}
\norm{A^k \bu} \ge c_A k^{m-1} \rho(A)^k\norm{\bu} 
\end{align*}
for sufficiently large $k\in\bN$.\\
\item If $\rho(A)=0$ then $A$ is a nilpotent matrix. In which case, both lower and upper bounds mentioned above hold trivially for any $\bu\in\reals^d$ for sufficiently large $k$.
\end{itemize}
\end{lemma}
\begin{proof} \label{proof:conv_rate_jord}
Let $P$ be a $d\times d $ invertible matrix such that 
\begin{align*}
P^{-1} A P =J
\end{align*}
where $J$ is a Jordan form of $A$, namely, $J$ is a block-diagonal matrix such that $J=\oplus_{i=1}^s J_{k_i}(\lambda_i) $ where $\lambda_1,\lambda_2,\dots, \lambda_s$ are eigenvalues of $A$, whose indices are $k_1,\dots,k_s$, respectively.  w.l.o.g we may assume that   $\absval{\lambda_1}=\rho(A)$ and that the corresponding index, which we denote by $m$, is maximal over all eigenvalues of maximal magnitude.
Let  $Q_1,Q_2,\cdots,Q_s$ and $R_1,R_2,\cdots,R_s$ denote partitioning of the columns of $P$ and the rows of $P^{-1}$, respectively, which conform with the Jordan blocks of $A$. \\ 
Note that for all $i\in[d]$, $J_{k_i}(0)$ is a nilpotent matrix of an order $k_i$. Therefore, for any $(\lambda_i,k_i)$ we have 
\begin{align*}
J_{k_i}(\lambda_i)^k &= (\lambda_i I_{k_i} + J_{k_i}(0) )^k  \\
&= \sum_{j=0}^k \binom{k}{j} \lambda_i^{k-j} J_{k_i}(0)^j\\
&= \sum_{j=0}^{k_i-1} \binom{k}{j} \lambda_i^{k-j} J_{k_i}(0)^j
\end{align*}
Thus,
\begin{align}
J_{k_i}(\lambda_i)^k/ (k^{m-1} \lambda_1^k )  
&= \sum_{j=0}^{k_i-1} \frac{\binom{k}{j} \lambda_i^{k-j} J_{k_i}(0)^j}{k^{m-1} \lambda_1^k} \nonumber\\
&= \sum_{j=0}^{k_i-1} \frac{\binom{k}{j} } {k^{m-1} }
\circpar{\frac{\lambda_i}{\lambda_1}}^k
 \frac{J_{k_i}(0)^j }{\lambda_i^j} \label{eq:jord1}
\end{align}
\\
The rest of the proof pivots around the following equality which holds for any $\bu\in\reals^{pd}$,
\begin{align}
\norm{A^k \bu } 
&= \norm{P J^k P^{-1}\bu} \nonumber \\
&= \norm{\sum_{i=1}^s Q_i J_{k_i}(\lambda_i)^k R_i \bu}  \nonumber\\
&= k^{m-1} \rho(A)^k \norm{\sum_{i=1}^s Q_i\circpar{ J_{k_i}(\lambda_i)/(k^{m-1} \lambda_1^k) }R_i \bu}
\end{align}
Plugging in \ref{eq:jord1} yields, 
\begin{align}
\norm{A^k \bu } 
&= k^{m-1} \rho(A)^k \norm{\underbrace{\sum_{i=1}^s Q_i\circpar{\sum_{j=0}^{k_i-1} \frac{\binom{k}{j} } {k^{m-1} }
\circpar{\frac{\lambda_i}{\lambda_1}}^k
 \frac{J_{k_i}(0)^j }{\lambda_i^j} }R_i \bu}_{\bw_k} } \label{eq:jord_main}
\end{align}
Let us denote the sequence of vectors in the l.h.s of the preceding inequality by $\{\bw_k\}_{k=1}^\infty$. Showing that the norm of $\{\bw_k\}_{k=1}^\infty$ is bounded from above and away from zero will conclude the proof. Deriving an upper bound is straightforward.
\begin{align}
\norm{\bw_k} &\le 
\sum_{i=1}^s  \norm{Q_i\circpar{\sum_{j=0}^{k_i-1} \frac{\binom{k}{j} } {k^{m-1} }
\circpar{\frac{\lambda_i}{\lambda_1}}^k
 \frac{J_{k_i}(0)^j }{\lambda_i^j} }R_i \bu}\nonumber \\
 &\le \norm{ \bu}
\sum_{i=1}^s  \norm{Q_i}\norm{R_i} \sum_{j=0}^{k_i-1}\norm{ \frac{\binom{k}{j} } {k^{m-1} }
\circpar{\frac{\lambda_i}{\lambda_1}}^k
 \frac{J_{k_i}(0)^j }{\lambda_i^j }} \label{eq:tmp2}
\end{align}
Since for all $i\in [d]$ we have
\begin{align*}
\frac{\binom{k}{j} } {k^{m-1} }
\circpar{\frac{\lambda_i}{\lambda_1}}^k \to 0 \quad \text{ or } \quad \frac{\binom{k}{j} } {k^{m-1} }
\circpar{\frac{\lambda_i}{\lambda_1}}^k \to1
\end{align*}
it holds that \ineqref{eq:tmp2} can be bounded from above by some positive scalar $C_A$. Plugging it in into \ref{eq:jord_main} yields
\begin{align*}
\norm{A^k \bu } \le C_A k^{m-1} \rho(A)^k \norm{\bu}
\end{align*}
\\
Deriving a lower bound on the norm of $\{\bw_k\}$ is a bit more involved. First, we define the following set of Jordan blocks which govern the asymptotic behavior of $\norm{\bw_k}$ 
\begin{align*}
\cI \eqdef  \myset{ i\in[s] }{ \absval{\lambda_i}=\rho(A) \text{ and }  k_i = m }
\end{align*}
\eqref{eq:jord1} implies that for all $i\notin\cI$ 
\begin{align*}
J_{k_i}(\lambda_i)^k / (k^{m-1} \lambda_1^k ) \to 0  \text{ as } k\to \infty. 
\end{align*}
As for $i\in\cI$, the first $k_i-1$ terms in \eqref{eq:jord1} tend to zero. The last term is a matrix whose entries are all zeros, except for the last entry in the first row which equals
\begin{align*}
\frac{\binom{k}{m-1} } {k^{m-1} }
\circpar{\frac{\lambda_i}{\lambda_1}}^k
1/(\lambda_i^{m-1}) 
&\approx
\circpar{\frac{\lambda_i}{\lambda_1}}^k
1/(\lambda_i^{m-1}) 
\end{align*}
By denoting the first column of each $Q_i$ by $q_i$ and the last row in each $R_i$ by $r_i^\top$, we get
\begin{align*}
\norm{\bw_k} &\approx 
\norm{\sum_{i\in\cI} \circpar{\frac{\lambda_i}{\lambda_1}}^k
\frac{1}{\lambda_i^{m-1}} Q_i J_{m}(0)^{m-1} R_i \bu}\\
&=
\norm{\sum_{i\in\cI} \circpar{\frac{\lambda_i}{\lambda_1}}^k
\frac{q_i r_i^\top  \bu}{\lambda_i^{m-1}}   }\\
\end{align*}
Now, if $\bu$ satisfies $r_1^\top\bu\neq0$ then since $q_1,q_2,\cdots,q_{|\cI|}$ are linearly independent, we see that the preceding can be  bounded from below by some positive constant $c_A>0$ which does not depend on $k$. That is, there exists $c_A>0$ such that $\norm{\bw_k}>c_A$ for sufficiently large $k$. Plugging it in into \eqref{eq:jord_main}  yields
\begin{align*}
\norm{A\bu} \ge c_A k^{m-1} \rho(A)^k\norm{\bu} 
\end{align*}
for any $\bu\in\reals^d$ such that $\inprod{\bu}{\br_1}\neq0$ and for sufficiently large $k$.
\end{proof}

The following is a well-known fact regarding \emph{Neuman series}, sum of powers of square matrices, which follows easily from \lemref{lem:conv_rate_jord}. 
\begin{lemma} \label{lem:conv_equi}
Suppose $A$ is a square matrix. Then, the following statements are equivalent:
\begin{enumerate}
\item $\rho(A)<1$. \label{cor:conv_equi_1}
\item $\lim_{k\to\infty}A^k=0$. \label{cor:conv_equi_2}
\item $\sum_{k=0}^\infty A^k$ converges. \label{cor:conv_equi_3}
\end{enumerate}
In which case, $(I-A)^{-1}$ exists and $(I-A)^{-1}=\sum_{k=0}^\infty A^k$.
\end{lemma}
\begin{proof}
First, note that all norms on a finite-dimensional space are equivalent. Thus, the claims stated in (\ref{cor:conv_equi_2}) and (\ref{cor:conv_equi_3}) are well-defined.\\
The fact that (\ref{cor:conv_equi_1}) and (\ref{cor:conv_equi_2}) are equivalent is a direct implication of \lemref{lem:conv_rate_jord}. Finally, the equivalence of (\ref{cor:conv_equi_2}) and (\ref{cor:conv_equi_3}) may be established using the following identity
\begin{align*}
(I-A)\sum_{k=0}^{m-1} A^k = I-A^{m},\quad m\in\bN
\end{align*}
\end{proof}

\subsubsection{Convergence Properties}
Let us now analyze the convergence properties of $p$-SCLI optimization algorithms. First, note that update rule (\ref{def:pscli_update_rule}) can be equivalently expressed as a single step rule by introducing new variables in some possibly higher-dimensional Euclidean space $\reals^{pd}$,
\begin{align}\label{Eq:canonical_dynamic}
\bz^0 =\circpar{\bx^0,\bx^1,\dots,\bx^{p-1}}^\top \in \reals^{pd}, \quad \bz^k =M(X)\bz^{k-1}+U N(X)\bb,\quad k=1,2,\dots
\end{align}
where 
\begin{align} \label{def:u}
U\eqdef  (\underbrace{0_d,\dots,0_d}_{p-1 \text{ times}}, I_d)^\top\in \reals^{pd\times d}
\end{align}
and where $M(X)$ is a mapping from $\reals^{d\times d}$ to $\reals^{pd\times pd}$-valued random variables which admits the following generalized form of companion matrices
\begin{align} \label{def:pcli}
\mymat{0_d & I_d  \\&0_d & I_d \\ && \ddots&\ddots\\&&& 0_d & I_d\\C_{0}(X)&&\dots&C_{p-2}(X)&C_{p-1}(X) }
\end{align} 
Following the convention in the field of linear iterative methods, we call $M(X)$ the \emph{iteration matrix}. Note that in terms of the formulation given in (\ref{Eq:canonical_dynamic}), consistency w.r.t $A\in\posdef{d}{\Sigma}$ is equivalent to 
\begin{align} \label{eq:conv_sol_z}
\bE \bz^k \to \underbrace{\circpar{-A^{-1}\bb,\dots,-A^{-1}\bb}^\top}_{p \text{ times}}
\end{align}
regardless of the initialization points and for any $\bb\in\reals^d$ and $\bz^0\in\reals^{pd}$. \\

To improve readability, we shall omit the functional dependency of the iteration, inversion and coefficient matrices on $X$ in the following discussion. Furthermore, \eqref{Eq:canonical_dynamic} can be used to derive a simple expression of $\bz^k$, in terms of previous iterations as follows
\begin{align*}
\bz^1 &= M^{(0)} \bz^0 + U N^{(0)} \bb\\
\bz^2 &= M^{(1)} \bz^1 + U N^{(1)} \bb =M^{(1)} M^{(0)} \bz^0 + M^{(1)}U N^{(0)} \bb+ U N^{(1)} \bb \\
\bz^3 &= M^{(2)} \bz^2 + U N^{(2)} \bb =M^{(2)}M^{(1)} M^{(0)} \bz^0 + M^{(2)}M^{(1)}U N^{(0)} \bb+ M^{(2)}U N^{(1)} \bb +U N^{(2)} \bb\\
\vdots\\
\bz^k &=  \prod_{j=0}^{k-1} M^{(j)}\bz^0 + \sum_{m=1}^{k-1} \prod_{j=m}^{k-1} M^{(j)}U N^{(m-1)} \bb  + U N^{(k-1)} \bb\\
&=  \prod_{j=0}^{k-1} M^{(j)}\bz^0 + \sum_{m=1}^{k} \circpar{ \prod_{j=m}^{k-1} M^{(j)} } U N^{(m-1)} \bb 
\end{align*}
where $\circpar{M^{(0)},N^{(0)}},\dots,\circpar{M^{(k-1)},N^{(k-1)}}$  are $k$ i.i.d realizations of the corresponding iteration matrix and inversion matrix, respectively. We follow the convention of defining an empty product as the identity matrix and defining the multiplication order of factors of abbreviated product notation as multiplication from the highest index to the lowest, i.e., $\prod_{j=1}^k M^{(j)} = M^{(k)}\cdots  M^{(1)} $. Taking the expectation of both sides yields
\begin{align}
\bE \bz^k &= \bE [M]^k \bz^0+ \circpar{\sum_{j=0}^{k-1} \bE [M]^j}  \bE [ U  N \bb] \label{EqLine:e25}
\end{align}
By \lemref{lem:conv_equi}, if $\rho(\bE M)<1$ then the first term in the r.h.s of \eqref{EqLine:e25} vanishes for any initialization point $\bz^0$, whereas the second term converges to $$(I-\bE M)^{-1}  \bE [ U   N \bb]$$ the fixed point of the update rule. On the other hand, suppose that $\circpar{\bE \bz^k}_{k=0}^\infty$ converges for any $\bz^0\in\reals^d$. Then, this is also true for $\bz^0=0$. Thus, the second summand in the r.h.s of \eqref{EqLine:e25} must converge. Consequently, the sequence $\bE[M]^k \bz^0$, being a difference of two convergent sequences, converges for all $\bz^0$, which implies $\rho(\bE [M])<1$.  This proves the following theorem.

\begin{theorem} \label{thm:conv_cli}
With the notation above, $\circpar{\bE \bz^{k}}_{k=0}^\infty$ converges for any $\bz^0\in\reals^d $ if and only if $\rho(\bE[ M])<1$. In which case, for any initialization point $\bz^0\in\reals^d$, the limit is 
\begin{align} \label{eq:limit_pscli}
	\bz^*  \eqdef \circpar{I-\bE M}^{-1}\bE [ U   N \bb]
\end{align} 
\end{theorem}

We now address the more delicate question as to how fast do $p$-SCLIs converge. To this end, note that by \eqref{EqLine:e25} and \thmref{thm:conv_cli} we have
\begin{align} \label{eq:conv_eq2}
\bE \left[\bz^k - \bz^*\right] \nonumber 
&= \bE [M]^k \bz^0+ \circpar{\sum_{l=0}^{k-1} \bE [M]^i}  \bE [ U   N \bb] -  \circpar{I-\bE M}^{-1}\bE [ U   N \bb]  \nonumber\\
&= \bE [M]^k \bz^0+   \circpar{I-\bE M}^{-1} \circpar{\circpar{I-\bE M}\sum_{l=0}^{k-1} \bE [M]^i - I }  \bE [ U   N \bb]  \nonumber\\
&= \bE [M]^k \bz^0-   \circpar{I-\bE M}^{-1} (\bE M)^k  \bE [ U   N \bb]   \nonumber \\
&= \bE [M]^k (\bz^0-   \bz^*  ) 
\end{align} 
Hence, to obtain a full characterization of the convergence rate of $\norm{\bE\rectpar{\bz^k-\bz^*}}$ in terms of $\rho(\bE M)$, all we need is to simply apply \lemref{lem:conv_rate_jord} with $\bE M$.

\subsubsection{Proof}
We are now in position to prove \thmref{thm:ic_cli}. Let  $\cA\eqdef(\syspola,N(X))$ be a $p$-SCLI algorithm over $\reals^d$, let $M(X)$ denote its iteration matrix and let $\quadab{A,\bb}$ be some quadratic function. According to the previous discussion, there exist $m\in\bN$ and $C(A),c(A)>0$ such that the following hold:
\begin{enumerate}
	\item For any initialization point $\bz^0\in\reals^{pd}$, we have that $(\bE \bz^k)_{k=1}^\infty$ converges to 
	\begin{align} \label{eq:conv_prop_pscli_limit}
	  \bz^*\eqdef \circpar{I-\bE M(A)}^{-1} \bE \left[ UN(A)\bb\right] 
\end{align}
\item For any initialization point $\bz^0\in\reals^{pd}$  and for any $h\in\bN$,
\begin{align} \label{eq:conv_prop_pscli_ub}
	\norm{\bE\rectpar{ \bz^k - \bz^*}} \le C_A k^{m-1} \rho(M(A))^k\norm{\bz^0-\bz^*} 
\end{align}
\item There exists $\br\in\reals^{pd}$ such that for any initialization point $\bz^0\in\reals^{pd}$ which satisfies $\inprod{\bz^0-\bz^*}{\br}\neq0 $ and sufficiently large  $k\in\bN$,
\begin{align} \label{eq:conv_prop_pscli_lb}
	\norm{\bE\rectpar{ \bz^k - \bz^*}} \ge c_A k^{m-1} \rho(M(A))^k\norm{\bz^0-\bz^*} 
\end{align}

\end{enumerate}
Since iteration complexity is defined over the problem space, we need to derive the same inequalities in terms of 
$$\bx^k= U^\top \bz^k$$
Note that by linearity we have $\bx^*= U^\top \bz^*$. For bounding $(\bx_k)_{k=1}^\infty$ from above we use (\ref{eq:conv_prop_pscli_ub}),
\begin{align}
\norm{\bE \left[\bx^k-\bx^*\right]} &=
\norm{\bE \left[U^\top\bz^k- U^\top\bz^*\right]} \nonumber \\
&\le \norm{U^\top} \norm{\bE \left[\bz^k-\bz^*\right]} \nonumber\\
&\le \norm{U^\top}C_A k^{m-1} \rho(M)^k\norm{\bz^0-\bz^*} \nonumber\\
&= \norm{U^\top}C_A k^{m-1} \rho(M)^k\norm{U\bx^0-U\bx^*} \nonumber\\
&\le \norm{U^\top}\norm{U}C_A k^{m-1} \rho(M)^k\norm{\bx^0-\bx^*} \label{ineq:x_up}
\end{align}
Thus, the same rate as in (\ref{eq:conv_prop_pscli_ub}), with a different constant, holds in the problem space. Although the corresponding lower bound takes a slightly different form, it proof is done similarly. Pick $\bx^0,\bx^1,\dots,\bx^{p-1}$ such that the corresponding $\bz^0$ is satisfied the condition in (\ref{eq:conv_prop_pscli_lb}). 
For sufficiently large $k\in\bN$, it holds that 
\begin{align} \label{ineq:x_lb}
\max_{k=0,\dots,p-1} \norm{\bE\bx^{k+j} - \bE\bx^*} 
&\ge \frac{1}{\sqrt{p}}\sqrt{\sum_{j=0}^{p-1} \normsq{\bx^{k+j} - \bx^*} }\nonumber \\
&=\frac{1}{\sqrt{p}} \norm{\bE \left[\bz^k\right]-\bz^*} \nonumber \\
&\ge \frac{c_A}{\sqrt{p}} k^{m-1} \rho(M)^k\norm{\bz^0-\bz^*} \nonumber\\
&= \frac{c_A}{\sqrt{p}} k^{m-1}\rho(M)^k\sqrt{\sum_{j=0}^{p-1} \normsq{\bx^{j} - \bx^*}}  
\end{align}

We arrived at the following corollary which states that the asymptotic convergence rate of any $p$-SCLI optimization algorithm is governed by the spectral radius of its iteration matrix. 
\begin{theorem} \label{thm:conv_rate_ulb}
Suppose $\cA$ is a $p$-SCLI optimization algorithm over $\posdefun{d}{\Sigma}$ and let $M(X)$ denotes its iteration matrix. Then, there exists $m\in\bN$ such that for any quadratic function $\quadab{A,\bb}\in\posdefun{d}{\Sigma}$ it holds that 
\begin{align*}
\norm{\bE \left[\bx^k-\bx^*\right]} = \bigO{k^{m-1}\rho(M(X))^k\norm{\bx^0-\bx^*}}
\end{align*}
where $\bx^*$ denotes the minimizer of $\quadab{A,b}$. Furthermore, there exists an initialization point $\bx^0\in\reals^d$, such that 
\begin{align*}
\max_{k=0,\dots,p-1} \norm{\bE\bx^{k+j} - \bE\bx^*}  = \Omega\circpar{\frac{k^{m-1}}{\sqrt{p}}\rho(M(X))^k\norm{\bx^0-\bx^*}}
\end{align*}
\end{theorem}

Finally, in the next section we prove that the spectral radius of the iteration matrix equals the root radius of the determinant of the characteristic of polynomial by showing that
\begin{align*}
	\det(\lambda I- M(X))&= \det(\syspola) 
\end{align*}
Combining this with the corollary above and by applying \ineqref{ineq:exp_x} and the like, concludes the proof for \thmref{thm:ic_cli}.
\subsubsection{The Characteristic Polynomial of the Iteration Matrix}
The following lemma provides an explicit expression for the characteristic polynomial of iteration matrices. The proof is carried out by applying elementary determinant manipulation rules. 
\begin{lemma} \label{lemma:mat_pol_spec}
Let $M(X)$ be the matrix defined in (\ref{def:pcli}) and let $A$ be a given $d\times d$ square matrix. Then, the characteristic polynomial of $\bE M(A)$ can be expressed as the following matrix polynomial 
\begin{align}
\chi_{\bE M(A)}(\lambda) =  (-1)^{pd} \det\circpar{\lambda^p I_d - \sum_{k=0}^{p-1} \lambda^{k} \bE C_k(A)}
\end{align}
\end{lemma}
\begin{proof}
As usual, for the sake of readability we omit the functional dependency on $A$, as well as the expectation operator symbol. For $\lambda\neq0$ we get, 
\begin{align*}
\chiM(\lambda) &= \det(M-\lambda I_{pd}) \\
&= \det\circpar{\begin{array}{cccc|c} -\lambda I_d & I_d &&& \\&-\lambda I_d & I_d&&\\ &&&&\\ && \ddots&\ddots\\&&& \\&&& -\lambda I_d & I_d\\ \hline C_{0}&&\dots&C_{p-2}&C_{p-1}-\lambda I_d \end{array}}\\
&= 
\det\circpar{\begin{array}{cccc|c} -\lambda I_d & I_d &&& \\&-\lambda I_d & I_d&&\\ &&&&\\ && \ddots&\ddots\\&&& \\&&& -\lambda I_d & I_d\\ \hline 0_d&C_{1}+\lambda^{-1}C_0&\dots&C_{p-2}&C_{p-1}-\lambda I_d \end{array}}\\
&= \det\circpar{\begin{array}{cccc|c} -\lambda I_d & I_d &&& \\&-\lambda I_d & I_d&&\\ &&&&\\ && \ddots&\ddots\\&&& \\&&& -\lambda I_d & I_d\\ \hline 0_d&0_d&
C_2 + \lambda^{-1}C_{1}+\lambda^{-2}C_0\dots &C_{p-2}&C_{p-1}-\lambda I_d \end{array}}
\end{align*}
\begin{align*}
&= \det\circpar{\begin{array}{cccc|c} -\lambda I_d & I_d &&& \\&-\lambda I_d & I_d&&\\ &&&&\\ && \ddots&\ddots\\&&& \\&&& -\lambda I_d & I_d\\ \hline 0_d&\dots &0_d&&
\sum_{k=1}^p \lambda^{k-p} C_{k-1} -\lambda I_d \end{array}}\\
&= \det(-\lambda I_d)^{p-1} \det\circpar{\sum_{k=1}^p \lambda^{k-p} C_{k-1} -\lambda I_d }
\end{align*}
\begin{align*}
&= (-1)^{(p-1)d} \det\circpar{\sum_{k=1}^p \lambda^{k-1} C_{k-1} -\lambda^p I_d}\\
&= (-1)^{pd} \det\circpar{\lambda^p I_d - \sum_{k=0}^{p-1} \lambda^{k} C_k}
\end{align*}
By continuity we have that the preceding equality also holds $\lambda=0$, as well.
\end{proof}

\subsection{Proof of \thmref{thm:conv_correct}} \label{section:conv_correct}
We prove that consistent $p$-SCLI optimization algorithms must satisfy conditions (\ref{consis_1_syspol}) and (\ref{consis_2_syspol}). The reverse implication is proven by reversing the steps of the following proof.\\

First, note that (\ref{consis_2_syspol}) is an immediate consequence of \corref{thm:conv_rate_ulb}, according to which  $p$-SCLIs converge if and only if the the root radius of the characteristic polynomial is strictly smaller than 1. As for (\ref{consis_2_syspol}), let $\cA\eqdef(\syspola,N(X))$ be a consistent $p$-SCLI optimization algorithm over $\posdefun{d}{\Sigma}$ and let $\quadab{A,\bb}\in\posdefun{d}{\Sigma}$ be a quadratic function. Furthermore, let us denote the corresponding iteration matrix by $M(X)$ as in (\ref{def:pcli}). By  \thmref{thm:conv_cli}, for any initialization point we have $$\bE \bz^k \to\circpar{I-\bE M(A)}^{-1} U \bE [N(A)]\bb $$ where $U$ is as defined in (\ref{def:u}), i.e.,
\begin{align*}
U\eqdef  (\underbrace{0_d,\dots,0_d}_{p-1 \text{ times}}, I_d)^\top\in \reals^{pd\times d}
\end{align*}
For the sake of readability we omit the functional dependency on $A$, as well as the expectation operator symbol. Combining this with \eqref{eq:conv_sol_z} yields
\begin{align*}
U^\top \circpar{I- M}^{-1} U N \bb= - A^{-1}\bb
\end{align*}
Since this holds for any $\bb\in\reals^d$, we get
\begin{align*}
U^\top \circpar{I-M}^{-1} U N= - A^{-1}
\end{align*}
Evidently, $N$ is an invertible matrix. Therefore,
\begin{align} \label{eq:conv_correct}
U^\top \circpar{I- M}^{-1} U = - ( N A) ^{-1} 
\end{align}
Now, recall that
\begin{align*}
M = \mymat{0_d & I_d &&& \\&0_d & I_d&&\\ &&&&\\ && \ddots&\ddots\\&&& \\&&& 0_d & I_d\\C_{0}&&\dots&C_{p-2}&C_{p-1} }
\end{align*} 
where $C_j$ denote the coefficient matrices. We partition $M$ as follows
\begin{align*}
\circpar{\begin{array}{c|cccc} M_{11}  &M_{12}\\ \hline M_{21}& M_{22} \end{array}} 
\eqdef\circpar{\begin{array}{cccc|c} 0_d & I_d &&& \\&0_d & I_d&&\\ &&&&\\ && \ddots&\ddots\\&&& \\&&& 0_d & I_d\\ \hline C_{0}&&\dots&C_{p-2}&C_{p-1} \end{array}}
\end{align*}
The l.h.s of \eqref{eq:conv_correct} is in fact the inverse of the Schur Complement of $I-M_{11}$ in $I-M$, i.e., 
\begin{align}
(I-M_{22} - M_{21}(I-M_{11})^{-1}M_{12})^{-1}   &= -(NA)^{-1}\nonumber\\
I-M_{22} - M_{21}(I-M_{11})^{-1}M_{12}   &= -NA \nonumber\\
M_{22} + M_{21}(I-M_{11})^{-1}M_{12}   &= I+NA \label{eqst1} 
\end{align}
Moreover, it is straightforward to verify that
\begin{align*}
\circpar{I-M_{11}}^{-1} = \mymat{I_d && I_d&I_d\\&I_d && I_d\\&&\ddots \\&&&I_d  }
\end{align*}
Plugging in this into (\ref{eqst1}) yields
\begin{align*} 
\sum_{i=0}^{p-1} C_i = I+NA
\end{align*}
equivalently, 
\begin{align} 
\syspol{}{1,A}= -NA
\end{align}
Thus concludes the proof.

\subsection{Proof of \lemref{lem:comp_poly}} \label{proof:lem:comp_poly}
First, we prove the following Lemma. Let us denote 
\begin{align*}
	q_r^*(z) \eqdef \circpar{z-(1-\sqrt[p]{r})}^p
\end{align*}
where $r$ is some non-negative constant.
\begin{lemma} \label{lem:eco_poly}
Suppose $q(z)$ is a monic polynomial of degree $p$ with complex coefficients. Then,
\begin{align*}
\rho(q(z))\le \absval{\sqrt[p]{\absval{q(1)}}-1} \iff q(z)=q_{\absval{q(1)}}^*(z)
\end{align*}
\end{lemma}


\begin{proof}
As the $\Leftarrow$ statement is clear, we prove here the only the $\Rightarrow$ part.\\
By the fundamental theorem of algebra $q(z)$ has $p$ roots. Let us denote these roots by $\zeta_1,\zeta_2,\dots,\zeta_p\in\bC$ . Equivalently,
\begin{align*}
q(z) = \prod_{i=1}^p (z-\zeta_i)
\end{align*}
Let us denote $r\eqdef\absval{q(1)}$. If $r\ge 1$ we get 
\begin{align} \label{ineq:ttt}
r &= \absval{\prod_{i=1}^p (1-\zeta_i)} = \prod_{i=1}^p \absval{1-\zeta_i} \le \prod_{i=1}^p (1+\absval{\zeta_i}) \nonumber\\&\le \prod_{i=1}^p (1+\absval{\sqrt[p]{r}-1}) 
= \prod_{i=1}^p (1+ \sqrt[p]{r}-1) =  r
\end{align}
Consequently, \ineqref{ineq:ttt} becomes an equality. Therefore, 
\begin{align} \label{eq:eq_roots_poly}
\absval{1-\zeta_i}=1+\absval{\zeta_i}=\sqrt[p]{r},\quad \forall i\in[p]
\end{align}
Now, for any two complex numbers $w,z\in\bC$ it holds that 
\begin{align*}
\absval{w+z}=\absval{w}+\absval{z} \iff \text{Arg}(w)=\text{Arg}(z)
\end{align*}
Using this fact in the first equality of \eqref{eq:eq_roots_poly}, we get that $\text{Arg}(-\zeta_i)=\text{Arg}(1)=0$, i.e., $\zeta_i$ are negative real numbers. Writing $-\zeta_i$ in the second equality of  \eqref{eq:eq_roots_poly} instead of $\absval{\zeta_i}$, yields $1-\zeta_i=\sqrt[p]{r}$, concluding this part of the proof.\\

The proof for $r\in[0,1)$ follows along the same lines, only this time we use the reverse triangle inequality,
\begin{align*}
r &= \prod_{i=1}^p\absval{ 1-\zeta_i} \ge 
\prod_{i=1}^p \circpar{1 - \absval{ \zeta_i} }
\ge 
\prod_{i=1}^p \circpar{1 - \absval{ \sqrt[p]{r}-1} }\\
&= \prod_{i=1}^p \circpar{1 - (1- \sqrt[p]{r})}=r 
\end{align*}
Note that in the first inequality, we used the fact that $r\in[0,1)\implies \absval{\zeta_i}\le1$ for all $i$.
\end{proof}

The proof for \lemref{lem:comp_poly} now follows easily. In case $q(1)\ge0$, if $q(z)= \circpar{z-(1-\sqrt[p]{r})}^p$ then, clearly,
\begin{align*}
	\rho(q(z))&=\rho\circpar{(z-(1-\sqrt[p]{r})^p)}= \absval{1-\sqrt[p]{r}}
\end{align*}
Otherwise, according to \lemref{lem:eco_poly}
\begin{align*}
	\rho(q(z))&>\absval{1-\sqrt[p]{r}}
\end{align*}
In case $q(1)\le0$, we must use the assumtpoin that the coefficients are reals (see Remark \ref{remark:coefficients_must_be_reals}), whereby the mere fact that
\begin{align*}
\lim_{z\in\reals,  z\to\infty} q(z) = \infty 
\end{align*}
combined with the Mean-Value theorem implies $\rho(q(z))\ge1$. This concludes the proof.
\begin{remark} \label{remark:coefficients_must_be_reals}
The requirement that the coefficients of $q(z)$ should be real  is inevitable.  To see why, consider the following polynomial, 
\begin{align*}
u(z)  = \circpar{z-\circpar{1-0.5e^{\frac{i\pi}{3}}}}^3
\end{align*}
Although $u(1)=\circpar{1-\circpar{1-1/2e^{\frac{i\pi}{3}}}}^3 = -1/8\le0$, it holds that $\rho(u(1))<1$. Indeed, not all the coefficients of $u(z)$ are real. Notice that the claim does hold for degree $\le3$, regardless of the additional assumption on the coefficients of $u(z)$.
\end{remark}

\subsection{Bounding the spectral radius of diagonal inversion matrices from below using scalar matrices} \label{ap:reduction}
We prove a lower bound on the convergence rate of $p$-SCLI optimization algorithm with diagonal inversion matrices. In particular, we show that for any $p$-SCLI optimization algorithm whose inversion matrix is diagonal there exists a quadratic function for which it does not perform better than $p$-SCLI optimization algorithms with scalar inversion matrix. We prove the claim for $d=2$. The general case follows by embedding the 2-dimensional case as a principal sub-matrix in some higher dimensional matrix in $\posdef{d}{[\mu,L]}$.\\

Let $\cA\eqdef(M(X),N(X))$ be a $p$-SCLI optimization algorithm and assume that $N(X)$ is a diagonal matrix. Define the following positive definite matrix
\begin{align}
B = \mymat{ \frac{L+\mu}{2} & \frac{L-\mu}{2} \\ \frac{L-\mu}{2} & \frac{L+\mu}{2}}
\end{align}
And note that $\spec{B}=\set{\mu,L}$. As usual, we wish to derive a lower bound on $\rho(M(B))$. To this end, denote 
\begin{align*}
N\eqdef N(B) = \mymat{\alpha & 0 \\ 0 & \beta}
\end{align*}
where $ \alpha,\beta\in\reals$. By a straightforward calculation we get that the eigenvalues of $-NB$ are
\begin{align} \label{eq:eigs_of_diagonal_inversion}
\sigma_{1,2}(\alpha,\beta) &=  \frac{-(\alpha+\beta)(L+\mu)}{4} \pm
\sqrt{ \circpar{\frac{(\alpha+\beta)(L+\mu)}{4}}^2 -  \alpha\beta L\mu}\nonumber\\
&=\frac{-(\alpha+\beta)(L+\mu)}{4} \pm
\sqrt{ (\alpha+\beta)^2 \frac{(L-\mu)^2}{16} +  \frac{1}{4} (\alpha-\beta)^2  L\mu}
\end{align}
Using similar arguments to the ones which were applied in the scalar case, we get that both eigenvalues of $-NB$ must be strictly positive as well as satisfy
\begin{align} \label{ineq:spec_diagonal}
\rho(M) \ge \min_{\alpha,\beta}\max\left\{ \absval{\sqrt[p]{\sigma_1(\alpha,\beta)}-1},\absval{\sqrt[p]{\sigma_2(\alpha,\beta) }-1}\right\}
\end{align}
\eqref{eq:eigs_of_diagonal_inversion} shows that the minimum of the preceding is obtained for $\nu=\frac{\alpha+\beta}{2}$, which simplifies to 
\begin{align*}
\max\left\{ \absval{\sqrt[p]{\sigma_1(\alpha,\beta)}-1},\absval{\sqrt[p]{\sigma_2(\alpha,\beta) }-1}\right\} &\ge 
\max\left\{ \absval{\sqrt[p]{\sigma_1(\nu,\nu)}-1},\absval{\sqrt[p]{\sigma_2(\nu,\nu) }-1}\right\}\\
&=
\max\left\{ \absval{\sqrt[p]{-\nu \mu}-1},\absval{\sqrt[p]{-\nu L }-1}\right\}
\end{align*}
The rest of the analysis is carried out similarly to the scalar case, resulting in 
\begin{align*}
\rho(M(B)) \ge \frac{\sqrt[p]{Q}-1}{\sqrt[p]{Q}+1}
\end{align*}

\end{appendices}

\end{document}